\documentclass[10pt, a4paper]{amsart}
\usepackage[utf8]{inputenc}
\usepackage{amsmath}
\usepackage{amssymb}
\usepackage{amsthm}
\usepackage{amsopn}
\usepackage{color}
\usepackage{enumerate}
\usepackage{nicefrac}
\usepackage{natbib}

\usepackage[hidelinks]{hyperref}

\newtheorem{theorem}{Theorem}[section]
\numberwithin{equation}{section}

\newtheorem{lemma}[theorem]{Lemma}
\newtheorem{definition}[theorem]{Definition}
\newtheorem{proposition}[theorem]{Proposition}
\newtheorem{corollary}[theorem]{Corollary}

\theoremstyle{remark}
\newtheorem{remark}[theorem]{Remark}

\newcommand{\Mcalp}{\mathcal{M}^+}
\newcommand{\Mcalm}{\mathcal{M}^-}

\begin{document}

\title[LOSS OF BOUNDARY CONDITIONS]{LOSS OF BOUNDARY CONDITIONS FOR FULLY NONLINEAR PARABOLIC EQUATIONS WITH SUPERQUADRADIC GRADIENT TERMS}
\author{ALEXANDER QUAAS, ANDREI RODRÍGUEZ}

\begin{abstract}
	We study whether the solutions of a fully nonlinear, uniformly para\-bolic equation with superquadratic growth in the gradient satisfy initial and homogeneous boundary conditions in the classical sense, a problem we refer to as the classical Dirichlet problem. Our main results are: the nonexistence of global-in-time solutions of this problem, depending on a specific largeness condition on the initial data, and the existence of local-in-time solutions for initial data $C^1$ up to the boundary. Global existence is know when boundary conditions are understood in the viscosity sense, what is known as the generalized Dirichlet problem. Therefore, our result implies loss of boundary conditions in finite time. Specifically, a solution satisfying homogeneous boundary conditions in the viscosity sense eventually becomes strictly positive at some point of the boundary.
\end{abstract}

\maketitle



\tableofcontents

\section{Introduction and main results}
The present article is a contribution to the study of qualitative properties of viscosity solutions of the so-called Cauchy-Dirichlet problem for the following fully nonlinear parabolic equation with superquadratic growth in the term with gradient dependence:
\begin{align}
	u_t -\Mcalm(D^2u) = |Du|^p &\quad\textrm{ in } \Omega\times(0,T) \label{modeleq},\\
	u = 0 &\quad\textrm{ on } \partial\Omega\times(0,T) \label{boundarydata},\\
	u(x,0) = u_0(x) &\quad\textrm{ in } \overline{\Omega} \label{initialdata},
\end{align}
where $\Omega\subset\mathbb{R}^n$ is a bounded domain satisfying both uniform interior and exterior sphere conditions. While this is not strictly necessary for all our results, it does establish a better connection between our main theorems. See Remarks \ref{intSphereLocal} and \ref{rmkextSphereUni}. We also assume $p>2$ throughout, except for certain remarks regarding the case $p\leq 2$ made in this introduction. See also Remark \ref{rmkValuePLocal}. Here $\Mcalm$ denotes one of Pucci's extremal operators, which are defined as follows: let $A, X \in S(n)$, the symmetric $n\times n$ matrices equipped with the usual ordering, $I$ denote the identity matrix, and $0 < \lambda < \Lambda$. Then
\begin{align*}
	\Mcalm(X) ={}&  \Mcalm(X, \lambda, \Lambda) = \inf \{\mathrm{tr}(AX) | \lambda I \leq A \leq \Lambda I\},\\
	\Mcalp(X) ={}&  \Mcalp(X, \lambda, \Lambda) = \sup \{\mathrm{tr}(AX) | \lambda I \leq A \leq \Lambda I\}.
\end{align*}
Alternatively, if we denote by $\lambda_i=\lambda_i(X)$ the eigenvalues of $X$, then
\begin{align*}
	\Mcalm(X) ={}& \lambda \sum_{\lambda_i > 0} \lambda_i + \Lambda \sum_{\lambda_i < 0} \lambda_i,\\
	\Mcalp(X) ={}& \Lambda \sum_{\lambda_i > 0} \lambda_i + \lambda \sum_{\lambda_i < 0} \lambda_i.
\end{align*}
Pucci's operators are fundamental to the study of fully nonlinear equations, at once acting as barriers to all equations sharing the same ellipticity constants (owing to the first definition) and allowing fairly explicit computations to be carried out (owing to the second). The Dirichlet condition \eqref{boundarydata} will be considered both in the classical sense and in the generalized sense of viscosity solutions. We will stress the distinction when necessary. Precise definitions and more on this later in this introduction. On the other hand, condition \eqref{initialdata} is always meant in the classical (pointwise) sense. See Remark \ref{nobottomLOBC}.

We assume the compatibility condition 
\begin{equation*}
	u_0(x) = 0 \quad \textrm{for all } x\in\partial\Omega
\end{equation*}
is also satisfied in the pointwise sense and that $u_0\in C^1(\overline{\Omega})$. As with the assumptions on $\Omega$, it is not strictly necessary to assume this regularity for $u_0$ throughout, but helps establish a connection between our main theorems. Also, we assume $u_0\geq 0$ without loss of generality, since \eqref{modeleq} is invariant with respect to additive constants.  

Equation \eqref{modeleq} can be seen as a generalization of the so-called viscous Hamil\-ton-Jacobi equation,
\begin{equation}\label{vhj}
	u_t - \Delta u = |D u|^p 	\textrm{ in } \Omega\times(0,T).
\end{equation}
For $p=2$, this corresponds to the deterministic Kardar-Parisi-Zhang equation, proposed by these authors in \cite{kardar1986dynamic} as a model for the profile of a growing interface. Mathematically, it is of interest because it is the simplest model of a parabolic equation with nonlinear dependence on the gradient, as well as a viscosity approximation of a first-order Hamilton-Jacobi equation (see \cite{evans1998partial}, Ch. 10).

For equation \eqref{vhj}, including all values $p>0$, it is well-know that there exists a unique, maximal-in-time classical solution $u\in C^{1, \alpha}(\Omega \times [0,T^*])$ for some $\alpha>0$ and $0<T^*\leq\infty$, assuming sufficient regularity for $\Omega$ and for the initial and boundary data (\cite{friedman2013partial}, Ch. 7). 

In \cite{souplet2002gradient} the nonexistence of global, classical solutions of problem \eqref{vhj}-\eqref{boundarydata}-\eqref{initialdata} is proved when $p>2$ and $u_0\in C^1(\overline{\Omega})$ and is suitably large. It is also shown here that this implies the occurrence of \emph{gradient blow-up} (GBU, for short). GBU is said to occur in finite time $0<T<\infty$ if a solution $u$ satisfies
\begin{equation*}
	\sup_{[0,T]\times\Omega} u < \infty, \quad \lim_{t \rightarrow T} \sup_{x\in\Omega} |D u (x,t)| = \infty.
\end{equation*}

A version of equation \eqref{vhj} containing a more general gradient term with superquadratic growth is studied in \cite{alaa1996solutions} in the context of weak solutions, for irregular initial data. The notable result is the nonexistence of global-in-time weak solutions with initial data $u_0$ a positive, bounded measure and suitably large.

In both \cite{alaa1996solutions} and \cite{souplet2002gradient}, the largeness condition on $u_0$ is (roughly speaking) given in terms of an $L^2$-product of $u_0$ with the principal eigenfunction of the Laplacian. The condition that appears in our Theorem \ref{mainonball} is essentially the same as the one in \cite{souplet2002gradient}. An alternative proof of global nonexistence for \eqref{vhj} given in \cite{quittner2007superlinear}, Theorem 40.2, uses a weaker condition on $u_0$. In this proof it is enough to consider the $L^q$-norm for any $q\geq 1$, but the argument does not adapt to more general nonlinearities.

There are different extensions of the results of \cite{souplet2002gradient}. Still in the context of classical solutions, the existence of global solutions and their large-time (or asymptotic) behavior for equation \eqref{vhj} with nontrivial right-hand side is studied in \cite{souplet2006global}. For equation
\begin{equation}\label{vhjsource}
	u_t - \Delta u = |D u|^p + \lambda h(x) \quad\textrm{ in } \Omega\times(0,T),
\end{equation}
where $\lambda\geq 0$, $h\in C^1(\overline{\Omega})$, $h\geq 0$, a complete description of the asymptotic behavior is given when $u_0, h$ are radially symmetric and $\Omega=B_R(0)$, for some $R>0$: in this case, for $h\not\equiv 0$, there exists a $\lambda^*>0$ such that
\begin{itemize}
	\item if $0\leq \lambda < \lambda^*$, then \eqref{vhjsource} has a global solution which converges to the solution of the steady-state equation
	\begin{equation}\label{steadysource}
		- \Delta v = |D v|^p + \lambda h(x) \quad\textrm{ in } \Omega,
	\end{equation}
	which additionally satisfies $v\in C^1(\overline{\Omega})$.
	\item if $\lambda = \lambda^*$, then $u$ converges to a solution $v\not\in C^1(\overline{\Omega})$ for any $u_0\in C^1(\overline{\Omega})$ with $u_0\leq v$. This implies GBU in \emph{infinite time}, i.e.,
	\begin{equation*}
		\limsup_{t\rightarrow\infty} \|D u(\cdot,t)\|_\infty = \infty.
	\end{equation*}
	\item if $\lambda > \lambda^*$, then \eqref{steadysource} has no solution and GBU in finite time occurs for \emph{any} $u_0\in C^1(\overline{\Omega})$. 
\end{itemize}
In the case of a general, bounded domain $\Omega\subset\mathbb{R}^n$ only a partial description is available.

Some of these results have been extended to equations with degenerate diffusion (i.e., with $\Delta_p$ in place of $\Delta$) in \cite{attouchi2012well} in the context of weak solutions. Other questions, such as determining precise blow-up rates, profiles and sets are addressed in \cite{yuxiang2010single}, \cite{attouchi2012well}, \cite{porretta2016profile}. See also \cite{quittner2007superlinear}, Ch. IV, and the references therein.

Equation \eqref{vhj} has also been studied from the viewpoint of viscosity solutions, in which a generalized notion of boundary conditions exists. The relevant phenomenon in this context is known as \emph{loss of boundary conditions} (LOBC, for short). More precisely, \eqref{boundarydata} is said to hold in the viscosity sense for \eqref{modeleq} if
\begin{align}
	\min\left(u_t - \Mcalm(D^2 u) - |Du|^p, u\right) \leq{}& 0, \quad \textrm{and}\\
	\max\left(u_t - \Mcalm(D^2 u) - |Du|^p, u\right) \geq{}& 0,
\end{align}
while loss of boundary conditions are said to occur whenever \eqref{boundarydata} is not satisfied in the classical sense. A standard reference for the concept of viscosity solutions is \cite{crandall1992user}; in particular see \cite{crandall1992user}, Sec. 7, which covers generalized boundary conditions. Another helpful reference for this last topic can be found in \cite{barles1994solutions}, Chap. 4., where its relation to the underlying optimal control problem is covered.

In \cite{barles2004generalized} it is proved that the Cauchy-Dirichlet problem for a class of fully-nonlinear equations which includes \eqref{modeleq} admits a unique, globally defined, continuous viscosity solution, assuming boundary conditions are understood in the viscosity sense. The result follows from a strong comparison principle proved by these authors and a subsequent application of Perron's method. This result is relevant in the case $p>2$, since it is shown in this same work that for $p\leq2$ there is no LOBC for either sub- or supersolutions, hence the classical comparison result of \cite{crandall1992user} applies, and global existence of solutions satisfying Dirichlet boundary conditions in the classical sense follows. A one dimensional example of LOBC is also provided which, in contrast to those furnished by our result, satisfies time-dependent boundary data. As these results apply directly to the problem under our consideration, we review some of them in Section \ref{comparison} for convenience.

Building on the existence of global solutions of the generalized Dirichlet problem, a natural question is to determine their large-time behavior. In this direction again there is an important distinction between the sub- and superquadratic cases, which are studied rather thoroughly in \cite{barles2010large} and \cite{tabet2010large}, respectively. Consider equation
\begin{equation}\label{vhjsourceplus}
	u_t - \Delta u + |Du|^p = f(x) \quad\textrm{ in } \Omega\times(0,T),
\end{equation}
where 
\begin{equation*}
	u(x,t) = \varphi(x) \quad\textrm{on } \partial\Omega\times(0,T),
\end{equation*}
is satisfied \emph{in the viscosity sense}, $f\in C(\Omega)$, $\varphi\in C(\partial\Omega)$, and $\varphi(x)=u_0(x)$ for all $x\in\partial\Omega$. In the superquadratic case, $p>2$, there are two possibilities: if the corresponding steady-sate equation
\begin{equation}\label{steadysourceplus}
	- \Delta v + |Dv|^p = f(x) \quad\textrm{ in } \Omega
\end{equation}
has a bounded subsolution, then there exists a solution $u_\infty$ of \eqref{steadysourceplus} and $u(x,t)\rightarrow u_\infty$ on $\overline{\Omega}$. If \eqref{steadysourceplus} fails to have bounded subsolutions, one must introduce the so-called \emph{ergodic problem} with state-constraint boundary conditions:
\begin{align}\label{ergodic}
	& -\Delta v + |Dv|^p = f(x) + c \quad\textrm{ in } \Omega,\\
	& -\Delta v + |Dv|^p \geq  f(x) + c \quad\textrm{ in } \partial\Omega.
\end{align}
Here $c\in\mathbb{R}$ is the so-called \emph{ergodic constant}, and is an unknown in problem \eqref{ergodic} together with $v$. Existence and uniqueness of solutions $(c,v)$ of \eqref{ergodic} are studied in \cite{lasry1989nonlinear}: $c$ is unique while $v$ is unique up to an additive constant. Convergence of $u(x,t) + ct$ to $v$ where $(c,v)$ is a solution of \eqref{ergodic}, as well as LOBC is then analyzed.

The behavior in the subquadratic case is more complicated. It depends also on whether $1<p\leq\nicefrac{3}{2}$  or $\nicefrac{3}{2}<p\leq2$ and becomes necessary to introduce the following problem, also studied in \cite{lasry1989nonlinear}, as an analogue of \eqref{steadysourceplus} and \eqref{ergodic}:
\begin{align}\label{stateconstraints}
	-\Delta v + |Dv|^p = f(x) + c &\quad\textrm{in } \Omega,\\
	v(x)\rightarrow\infty &\quad\textrm{as } x\rightarrow\partial\Omega.
\end{align}
We refer the reader to \cite{barles2010large}.

A different type of result concerning large-time behavior is given in \cite{porretta2012null}. It is shown that there exist constants $K, \lambda$ and $C$ such that the solution of the generalized Dirichlet problem for \eqref{vhjsourceplus} with homogeneous boundary data and \emph{any} compatible initial data $u_0\in C(\overline{\Omega})$ satisfies, for every $t\geq K \|u_0\|_\infty$,
\begin{equation*}
	u(\cdot,t)\in W^{1,\infty}(\Omega), \quad\textrm{and}\quad \|u(\cdot,t)\|_\infty + \|D u(\cdot, t)\|_\infty \leq Ce^{-\lambda t}.
\end{equation*}
In particular, after some finite time, the solution $u$ satisfies the boundary data in the classical sense. This property is then applied to the  interesting problem of the null controllability of \eqref{vhjsourceplus}.

Regarding regularity of solutions, it is proved in \cite{capuzzo2010holder} that if $u$ is a bounded, upper-semicontinuous viscosity subsolution of the (possibly degenerate) elliptic equation
\begin{equation}\label{ellipticvhj}
	-\mathrm{tr}(A(x)D^2u) + \lambda u + |Du|^p = f(x) \quad\textrm{for all } x\in\Omega, 
\end{equation}
where $p>2$, $\Omega\subset \mathbb{R}^n$ is a regular domain, $\lambda>0$, and $A:\Omega\to S(n)$ and $f$ satisfy fairly standard assumptions, then $u$ is \emph{globally} Hölder continuous with exponent $\alpha = \nicefrac{p-2}{p-1}$ (i.e., $u\in C^{0,\nicefrac{p-2}{p-1}}(\overline{\Omega})$). As noted in \cite{barles2010short}, the result above is surprising, since most regularity results apply to actual \emph{solutions} of \emph{uniformly} elliptic equations that satisfy \emph{subquadratic} growth conditions, none of which points are met in the assumed hypotheses. The authors of \cite{capuzzo2010holder} go on to prove interior Lipschitz bounds for \emph{solutions} of \eqref{ellipticvhj} by the so-called \emph{weak-Bernstein} method introduced in \cite{barles1991weak}. These results are valid for fully-nonlinear equations satisfying hypotheses which are discussed in detail in \cite{barles2010short}.

In \cite{barles2010short} a slight simplification of the proof of Hölder regularity is provided, and the relation to the solvability of Dirichlet problem is analyzed. In short, if a general boundary condition $u=\varphi$ with $\varphi\in C(\partial\Omega)$ is assumed in the viscosity sense, an additional reason for the occurrence of LOBC is that $\varphi$ might not have the same regularity as $u$. This is, of course, irrelevant to the case of homogeneous boundary data.

Time-dependent versions of these regularity results are proven in \cite{armstrong2015viscosity}, though they require the additional assumption that
\begin{equation*}
	u_t \geq -C \quad \textrm{for all } (x,t) \in \Omega\times(0,T)
\end{equation*}
for some $C\geq0$ be satisfied \emph{in the viscosity sense}. This means that: for all $(x,t) \in  \Omega\times(0,T)$, if $(a,\xi,X)\in \mathcal{P}^{2,+} u(x,t)$, the parabolic superjet of a subsolution $u$ (see, e.g., \cite{crandall1992user} for definitions), then $a\geq -C$. To the best of our knowledge, there is no readily available result in the context of viscosity solutions that would allow us to do away with this assumption, which is why we have followed the strategy of regularizing the solution (see Sec. \ref{regularization}).\\[2pt]

\noindent\textit{Main results}\\

Our main results are the following. We begin by proving the existence of solutions of \eqref{modeleq}-\eqref{boundarydata}-\eqref{initialdata} that for a small time satisfy the boundary data in the classical sense. The existence time depends only on a gradient bound for the initial data, the remaining constants usually considered universal. 
\begin{theorem}\label{localexthm}
	Let $u_0 \in C^1(\overline{\Omega})$. There exists a $T^*>0$, depending only on $\Lambda, \lambda, n, \Omega$ and $\|u_0\|_{C^1(\overline{\Omega})}$, such that the viscosity solution of \eqref{modeleq} in $\Omega\times(0,T^*)$ satisfies \eqref{boundarydata} and \eqref{initialdata} in the classical sense.
\end{theorem}
Since we already have the existence result of \cite{barles2004generalized}, we need only show that \eqref{boundarydata} is satisfied is the classical sense. For this we use a barrier argument, following the construction of comparison functions used in \cite{attouchi2012well} to show local existence under a slightly different strategy.

Next we prove the nonexistence of global solutions to the classical Dirichlet problem when $\Omega=B_1(0)$ and the initial data is radially symmetric, and suitably large. Again, thanks to the global existence result of \cite{barles2004generalized}, this implies the occurrence of LOBC.
\begin{theorem}\label{mainonball}
	Let $u_0\in C^1(\overline{B_1(0)})$ be a radial function. Then, there exist positive constants $\delta=\delta(\lambda, \Lambda, n)$ and $M=M(\lambda, \Lambda, n, p)$ such that, if
	\begin{equation}\label{intconditionu0}
		\int_\delta^{1-\delta} u_0(r) \,dr > M
	\end{equation}	
	then the solution $u$ of \eqref{modeleq}-\eqref{boundarydata}-\eqref{initialdata} with $\Omega = B_1(0)$ and initial data $u_0$ has LOBC at some finite time $T=T(u_0)$.
\end{theorem}
The proof of Theorem \ref{mainonball} uses key ideas from that of Theorem 2.1 in \cite{souplet2002gradient}. The main difficulty in adapting this proof is its crucial use of the divergence structure of the Laplacian by repeatedly using integration by parts. We remedy this problem by using the divergence form of the Pucci operator, available for radial solutions (see, e.g., \cite{felmer2004positive}), and the regularization by $\inf$-$\sup$-convolution introduced in \cite{lasry1986remark}. Combining these techniques we obtain an equation in divergence form which is satisfied point-wise and all of whose terms are integrable. Afterwards, the main complications are keeping track of the terms which depend on the regularization parameters and providing estimates which are independent of these. We also adapt a weighted, one-dimensional version of Poincaré's inequality, and make use of different results from \cite{esteban2005nonlinear} and \cite{esteban2010eigenvalues} regarding the principal eigenvalue problem for the Pucci operator.

This result is extended to show LOBC occurs for solutions of \eqref{modeleq} in a sufficiently regular bounded domain in Corollary \ref{gendomain}, and then to equations with more general nonlinearities, first in the radially symmetric case, then also for a bounded domain as above. The most general result is the following. Consider
\begin{equation}\label{uniparaboleq}
	u_t - F(D^2 u) = f(Du) \textrm{ in } \Omega \times (0,T),
\end{equation}
where $F:S(n) \rightarrow \mathbb{R}$ is uniformly elliptic, i.e., 
\begin{equation}\label{ourellipticity}
\Mcalm(X - Y) \leq F(X) - F(Y) \leq \Mcalp (X - Y) \quad \textrm{for all } X,Y \in S(n),
\end{equation}
and vanishes at zero, i.e., $F(0) = 0$, and $f:\mathbb{R}^n \rightarrow \mathbb{R}$ satisfies $f(\xi) \geq |\xi|^2 h(|\xi|)$ for all $\xi\in \mathbb{R}^n$, where $h:\mathbb{R}\rightarrow\mathbb{R}$ is positive, nondecreasing, grows more slowly than any positive power, and is such that $\xi \mapsto |\xi|^2 h(|\xi|)$ is convex. Precise hypotheses on $h$ are given in Section \ref{extensions}.
\begin{theorem}\label{uniparabolthm}
	Assume that $F$, $f$, and $h$ are as described above. If additionally $h$ satisfies
	\begin{equation}\label{growthh}
		\int_1^\infty \frac{1}{s h(s)} \,ds < \infty,
	\end{equation}
	then there exists $u_0 \in C^1(\overline{\Omega}),$ with $u_0\geq 0$ and $u_0|_{\partial \Omega} = 0$, such that LOBC occurs for solutions of \eqref{uniparaboleq}-\eqref{boundarydata}-\eqref{initialdata} in some finite time $T=T(u_0)$.
\end{theorem}
This result follows more or less easily from Theorem \ref{mainonball} and the main ideas used in its proof, as do the other extensions given in the final section.

The organization of the article is as follows. In Section \ref{comparison} we briefly review the results of \cite{barles2004generalized} which are directly used in our work. Section \ref{localex} is devoted to the proof of Theorem \ref{localexthm}. In Section \ref{techsec} we gather the technical results which lead us to the approximate equation we use to prove the nonexistence result, as well as some fundamental facts and estimates related to the eigenvalue problem for the Pucci extremal operator in a radial case. The statements and remarks of this section contain key concepts and notation used in the proof of the Theorem \ref{mainonball} and its subsequent generalizations. Section \ref{nonexistence} contains the proof of our main result in the radially symmetric case, Theorem \ref{mainonball}, and its extension to a bounded domain. This is the core of our work. Finally, in Section \ref{extensions} we provide extensions to more general equations, including the proof of Theorem \ref{uniparabolthm}.

After completing a first version of this work, we learned that the occurrence of LOBC, together with other closely related results, had been obtained for the equation involving the Laplacian (i.e., for \eqref{vhj}) in \cite{porretta2017analysis}.

\section{Comparison, existence and uniqueness}\label{comparison}

Existence and uniqueness for the so-called generalized Dirichlet problem for 
\begin{equation} \label{geneqbarles}
	u_t + G(x, t, Du, D^2u) = 0 \quad\textrm{ in } \Omega\times(0,T),
\end{equation}
where the boundary condition
\begin{equation}\label{genboundarydata}
	u = g \quad\textrm{ on }\partial\Omega\times(0,T),
\end{equation}
$g\in C(\partial\Omega\times (0,T))$, is understood in the viscosity sense, is proven in Theorem 5.1 in \cite{barles2004generalized}. For convenience, in this section we quote the main results of this work, as well as a couple of remarks relevant to our purposes. Here $G$ is a continuous function that satisfies the degenerate ellipticity condition,
\begin{equation}\label{theirellipticity}
	G(x,t,\xi,X) \leq G(x,t,\xi,Y) \quad \textrm{if } X \geq Y,
\end{equation}
for all $x\in\overline{\Omega}$, $t\in[0,T]$, $\xi\in \mathbb{R}^n$ and $X,Y\in S(n)$, together with two key hypothesis, for which we must introduce additional notation. Note that condition \eqref{theirellipticity} uses the opposite sign convention than the one used in \eqref{ourellipticity}. See also the discussion at the beginning of Subsection \ref{comparisonsubsec}.

Let $h_1:[0,\infty) \rightarrow [0,\infty)$ be a continuous function. We say $h_1$ satisfies property (P) if the following hold:
\begin{enumerate}[(i)]
	\item\label{P1} $\int_1^\infty \frac{s}{h_1(s)} \,ds < \infty$,
	\item\label{P2} for any $C>0$, $s$ large enough and $L\geq1$, the map $L\mapsto h_1(Ls) - CL^2h_1(s)$ is increasing,
	\item\label{P3} for any $C, \tilde{C}>0$, there exists $\bar{s}>0, \bar{L}\geq 1$ such that
		\begin{equation}
			h_1(Ls) - CL^2h_1(s) \geq \tilde{C}Ls \quad \textrm{for } s\geq\bar{s}, L\geq \bar{L}.
		\end{equation}
\end{enumerate}

The key assumptions on $G$ as the following:
\begin{enumerate}[(H1)]
	\item\label{barlesH1} There exists constants $C_1, C_2>0$ and a continuous function $h_1$ satisfying property (P) such that, for all $x\in\overline{\Omega}$, $t\in[0,T]$, $\xi\in \mathbb{R}^n$ and $X\in S(n)$, we have
	\begin{equation}
		G(x,t,\xi,X) \geq -C_1 -C_2\|X\| + h_1(|\xi|).
	\end{equation}
	\item\label{barlesH2} For any $\epsilon>0$, there exists $0<\mu_\epsilon<1$ converging to $1$ as $\epsilon\rightarrow 0$ such that
		\begin{equation*}
			G(y, s, \xi_2, Y) - G(x,t, \mu_\epsilon^{-1}\xi_1, \mu_\epsilon^{-1}X) \leq o(1)
		\end{equation*}
		for all $x,y\in\overline{\Omega}, t,s,\in[0,T], \xi_1, \xi_2\in\mathbb{R}^n$ and for all $X,Y\in S(n)$ satisfying the following properties for some $K>0$ and a sufficiently small $\eta>0$:
		\begin{align*}
			&-\frac{K\eta}{\epsilon^2}I_{2n} \leq \left(\begin{array}{cc} X & 0 \\ 0 & -Y \end{array}\right) \leq \frac{o(1)}{\epsilon^2}\left(\begin{array}{cc} I_n & -I_n \\ -I_n & I_n \end{array}\right) + o(1)I_{2n},\label{matrixineq}\\
			&|\xi_1 - \xi_2|\leq K\epsilon\min\{|\xi_1|,|\xi_2|\},\\
			&|x-y|+|t-s|<\epsilon.
		\end{align*}
\end{enumerate}

The main result is the following:

\begin{theorem}[Strong Comparison Result]\label{SCR}
	Assume $u_0\in C(\overline{\Omega})$, and let $u$ and $v$ be respectively a bounded upper-semicontinuous (USC, for short) supersolution and a bounded lower-semicontinuous (LSC) supersolution of \eqref{geneqbarles}-\eqref{genboundarydata}-\eqref{initialdata}, where $G$ satisfies hypotheses \emph{(H1)} and \emph{(H2)}. Then $u\leq v$ in $\Omega\times [0,T]$. Moreover, if we define $\tilde{u}$ on $\overline{\Omega}\times[0,T]$ by setting
	\begin{equation}\label{redef}
		\tilde{u}(x,t) = \left\{ \begin{array}{cl}
			\displaystyle{\limsup_{\substack{(y,s)\rightarrow(x,t) \\ (y,s)\in\Omega\times (0,T)}}} u(y,s) & \textrm{ on }\partial\Omega \times (0,T],\\[18pt]
			u(x,t) & \textrm{ otherwise,}
		\end{array}\right.
	\end{equation}
	and similarly define $\tilde{v}$, then $\tilde{u}$ and $\tilde{v}$ are still respectively a bounded USC subsolution and a bounded LSC supersolution of  \eqref{geneqbarles}-\eqref{genboundarydata}-\eqref{initialdata} and $\tilde{u} \leq \tilde{v}$ in $\overline{\Omega}\times[0,T]$.
\end{theorem}

As is standard, existence is proven by combining this result with Perron's method of sub- and supersolutions.

\begin{remark}\label{comparisonuptoboundary} 
	When Theorem \ref{SCR} is used to compare \emph{continuous} sub- and supersolutions, comparison holds up to the boundary without having to redefine the functions as in \eqref{redef}.
\end{remark}

\begin{remark}\label{signchange} The lower bound of \eqref{barlesH1} implies that the gradient nonlinearity has the opposite sign to that of \eqref{modeleq}. However, the results proved for
	\begin{equation*}
		u_t - \Mcalp(D^2u) + |Du|^p = 0 \quad \textrm{ in } \partial\Omega\times(0,T)
	\end{equation*}
are valid for \eqref{modeleq} provided we exchange the role of sub- and supersolutions. Indeed, $u$ is a subsolution of the above equation if and only if $-u$ is a supersolution of \eqref{modeleq}. This is already noted in Remark 3.2 of \cite{barles2004generalized}. We note also that in \cite{barles2004generalized} there is no requirement that the solution be nonnegative, as there is in the proofs of gradient blow-up given in \cite{souplet2002gradient}.
\end{remark}

We will verify that hypotheses (H\ref{barlesH1}) and (H\ref{barlesH2}) apply to the equations considered in this work (after the appropriate sign change) in Section \ref{extensions}.

\begin{remark}\label{nolobcforsubs}
	Following the exchange of sub- and supersolutions mentioned in the previous remark, it follows from Proposition 3.1 in \cite{barles2004generalized} that any supersolution $v$ of \eqref{modeleq} satisfies $v\geq0$ on $\partial\Omega\times (0,T)$ in the classical sense for any given $T>0$. Hence, if LOBC occurs, as we prove later, then the solution satisfying \eqref{boundarydata} in the generalized sense must become strictly positive at some point of the boundary.
\end{remark}

\begin{remark}\label{nobottomLOBC}
	As mentioned in the introduction, the initial condition \eqref{initialdata} is always meant in the classical sense. There is no loss of generality in this assumption. It is a consequence of Lemma 4.1 in \cite{lio2002comparison} that there is no LOBC on the bottom of the parabolic domain, $\overline{\Omega}\times\{t=0\}$.
\end{remark}

\begin{remark}\label{uniformboundu}
	An easy but important consequence of the comparison result is that solutions $u$ of \eqref{modeleq}-\eqref{initialdata} are uniformly bounded and nonnegative. Indeed, for $u_0\geq0$, $\underline{v}\equiv 0$ and $\bar{v}\equiv \sup_{\overline{\Omega}} u_0$ are respectively sub- and supersolutions, so by comparison we have
	\begin{equation}\label{boundu}
		0 \leq u(x,t) \leq \sup_{\overline{\Omega}} u_0 \quad\textrm{for all } x\in \overline{\Omega}, 0\leq t\leq T.
	\end{equation}
	In particular, $\|u\|_\infty \leq \|u_0\|_\infty$.
\end{remark}

\section{Existence of local solutions}\label{localex}

We follow the construction of the comparison functions used to prove local existence of solutions for a related problem in \cite{attouchi2012well}, accounting for the presence of the extremal operators and providing additional detail regarding the choice of constants. 

\textit{Proof of Theorem \ref{localexthm}:} \emph{Step 1: A time-independent barrier.} We define a time-independent comparison function in a neighborhood of a fixed $x_0\in\partial\Omega$, using the exterior sphere condition, and prove that it is a supersolution of \eqref{modeleq}. We will address the initial and boundary conditions in a later step.

From the exterior sphere condition there exists a ball of radius $\rho>0$ centered at $x_1\notin\overline{\Omega}$, tangent to $\partial \Omega$ at $x_0$. We will employ the radial variables ${r= |x - x_1|}$, where ${\rho<r<\rho+\eta}$ for some $\eta>0$, and ${s = |x - x_1| - \rho}$. In this and the following steps we will compare the solution $u$ to different functions in the set 
\begin{equation*}
	\Gamma = \{x\in\Omega\ | \ 0<s = |x-x_1| - \rho < \eta \}.
\end{equation*}

Let $\varphi(s) = s(s+\mu)^{-\beta}$ with $\mu, \beta>0$ to be chosen later, and define 
\begin{equation}
	\bar{v}(x) = \varphi(|x-x_1|-\rho) = \varphi(s). 
\end{equation}

For any $C^2$ radial function, say $\phi(x) = \phi(|x|)$, a standard computation of the eigenvalues of $D^2\phi$ at any point gives them explicitly as $\phi''$ and $\nicefrac{\phi'}{|x|}$ with multiplicities $1$ and $n-1$, respectively, where $'$ denotes the derivative in the radial direction. By the definition of the extremal operator, this gives
	
	\begin{equation}\label{radialpucci}
		\Mcalm(D^2\phi) = \min_{a,b\in\{\lambda, \Lambda\}}(a\phi'' + b\frac{n-1}{r}\phi').
	\end{equation}

Setting $\beta<1$, we compute
\begin{align*}
	&\varphi '(s) = [(1-\beta)s + \mu](s+\mu)^{-\beta-1} > 0,\\
	&\varphi ''(s) = -\beta[(1-\beta)s + 2\mu](s+\mu)^{-\beta-2} < 0.
\end{align*} 
Hence, the extremal operator takes the form
\begin{align*}
	\Mcalm(D^2\bar{v})(s) ={}& \Lambda\varphi''(s) + \lambda \left(\frac{n-1}{s + \rho}\right)\varphi'(s)\\
	={}& -\Lambda\beta[(1-\beta)s + 2\mu](s+\mu)^{-\beta-2}\\
	& + \lambda \left(\frac{n-1}{s+\rho}\right)[(1-\beta)s+\mu](s+\mu)^{-\beta-1}.
\end{align*}
The function $\bar{v}$ is a supersolution if
\begin{equation*}
	 -\Mcalm(D^2\bar{v}) \geq |\nabla\bar{v}|^p = |\varphi'|^p.
\end{equation*}
That is, from the previous computations, if
\begin{align*}
	\left(\Lambda\beta[(1-\beta)s + 2\mu] - \lambda \left(\frac{n-1}{s + \rho}\right)[(1-\beta)s + \mu](s+\mu)\right)(s+\mu)^{-\beta-2}\\ 
	\qquad \geq [(1-\beta)s + \mu]^p(s+\mu)^{-p(\beta+1)}.
\end{align*}
Here we have factored the leading term $(s+\mu)^{-\beta-2}$ in the left-hand side. We proceed to show that its coefficient $K$ is positive for the right choices of $\mu$ and $\beta$.

Setting $\eta = \mu$ and using only that $0<\beta<1$ and $0<s<\eta=\mu$, we have
\begin{equation*}
	K > 2\Lambda\beta\mu - \lambda\left(\frac{n-1}{s + \rho}\right)((1-\beta)\mu + \mu)2\mu.
\end{equation*}
Hence, to have $K>0$ it is sufficient that
\begin{equation}
	\mu < \frac{\beta}{2-\beta}\left(\frac{2\Lambda\rho}{\lambda(n-1)}\right).
\end{equation}
Next, we verify that
\begin{equation}\label{dom_pow}
	K (s+\mu)^{-\beta-2} \geq [(1-\beta)s + \mu]^p(s+\mu)^{-p(\beta+1)}.
\end{equation}
Again $0<\beta<1$ implies
\begin{equation*}
	[(1-\beta)s + \mu]^p \leq (s + \mu)^p,
\end{equation*}
then
\begin{align*}
	[(1-\beta)s + \mu]^p(s+\mu)^{-p(\beta+1)} \leq{}& (s+\mu)^p(s+\mu)^{-p-p\beta}\\
	={}& (s+\mu)^{-p\beta}.
\end{align*}

Hence \eqref{dom_pow} holds if $(s+\mu)^{-p\beta} \leq K(s+\mu)^{-\beta - 2}$, that is, if 
\begin{equation}
	K^{-1} \leq (s + \mu)^{\beta(p-1) - 2}.
\end{equation}

Setting $\beta<\frac{1}{2(p-1)}$ gives $\beta(p-1) - 2 < -\frac{3}{2}$, so that the term on the right is singular. This precise value of $\beta$ will be useful in a moment. Using once more that $0<s<\mu$, it is sufficient to have
\begin{equation}\label{delicate_mu}
	K^{-1} \leq (2\mu)^{\beta(p-1) - 2}.
\end{equation}

We recall that $K$, the coefficient defined above, also depends on $\mu$. However, from the above computations we have that for small $\mu$,
\begin{equation*}
	K\geq  2\Lambda\beta\mu - \lambda\left(\frac{n-1}{s + \rho}\right)((1-\beta)\mu + \mu)2\mu \geq C_1\mu - C_2\mu^2,
\end{equation*}
hence $K^{-1}= O(\mu^{-1})$ as $\mu\rightarrow 0$, whereas the previous choice for $\beta$ gives that the right-hand side of \eqref{delicate_mu} is $O(\mu^{-\frac{3}{2}})$. Therefore, choosing $\mu$ small enough gives all the desired inequalities.

\emph{Step 2: Time-dependent control.} We introduce a second comparison function which will help us relate the solution $u$ of \eqref{modeleq}-\eqref{boundarydata}-\eqref{initialdata} to the supersolution $\bar{v}$ constructed in the previous step. 

Let 
\begin{equation*}
	\bar{u}(x,t) = At + C(1- e^{-\gamma s}),
\end{equation*}
and write $\psi(s) = 1 - e^{-\gamma s}$. We will prove that for appropriate choices of the positive constants $A, C$ and $\gamma$, $\bar{u}$ satisfies
\begin{align} %%% XXX ALIGN? LABELS? %%%
	\bar{u}_t - \Mcalm(D^2 \bar{u}) \geq |D\bar{u}|^p &\quad\textrm{ in } \Gamma \times (0,\infty),\label{barusuper}\\
	\bar{u} \geq u &\quad\textrm{on } (\partial\Gamma\cap\Omega) \times (0,\infty),\label{barusuperboundary}\\
	\bar{u} \geq u_0 &\quad\textrm{ on } \overline{\Gamma}\times\{t=0\}\label{barubottom},
\end{align}
where \eqref{barusuperboundary} and \eqref{barubottom} are meant in the classical sense. 

Denote by $\nu$ the exterior unit normal at $x_0\in\partial\Omega$. For $x= x_0 - s\nu,\ t=0$, this is
\begin{equation*}
	\bar{u}(x,0) = C(1- e^{-\gamma s})\geq u_0 (x).
\end{equation*}

We use that
\begin{equation}\label{psiprime}
	0 < \left|\frac{\partial \bar{u}}{\partial \nu}(x_0)\right| = C\psi'(0) = C\gamma < \infty,
\end{equation}
$\|Du_0\|_\infty < \infty$, $u_0(x_0) = \psi(0) = 0$, and that both $u_0$ and $\bar{u}$ are non-negative to choose $C>0$ large enough, so that for small $s$, say $0<s<\delta$, we have
\begin{equation*}
	u_0(x) = u_0(x_0 - s\nu) \leq C\psi(s).
\end{equation*}
In other words, we are comparing the first-order expansions in the direction $-\nu$. Then, for $\delta \leq s \leq \eta$, we may also take
\begin{equation}\label{CversusGamma}
	C \min \{1, \min_{\delta\leq s \leq \eta} \bar{u} (x_0 - s\nu,0) \}> \max_{\bar{\Omega \times [0,T]}} u_0,
\end{equation}
since the minimum above is strictly positive. We may repeat this reasoning in the other directions which sweep $\overline{\Gamma}$, by considering an extension by zero of $u_0$ to the corresponding section of the annular domain where $\psi$ is defined. The choice of $C$ remains bounded since $\overline{\Gamma}$ is compact. Furthermore, it can be chosen uniformly with respect to $x_0$ since $\overline{\Omega}$ is compact. Observe that this also ensures that $\bar{u}\geq u$ on the rest of $\partial_p (\Gamma \times (0,\infty))$.

To check \eqref{barusuper}, we compute
\begin{equation}\label{puccipsi}
	-\Mcalm(D^2 \bar{u}) = \left(\Lambda\gamma - \lambda\frac{n-1}{s + \rho}\right)C\gamma e^{-\gamma s},
\end{equation}
and observe that choosing $\gamma$ large enough gives $-\Mcalm(D^2\bar{u})\geq 0$. Then, $\bar{u}$ is a supersolution if we can get
\begin{equation*}
	\bar{u}_t \geq |D\bar{u}|^p.
\end{equation*}
This amounts to taking
\begin{equation*}
	A \geq \max_{0\leq s \leq \eta} (C\gamma e^{-\gamma s})^p.
\end{equation*}

We have therefore proved that $\bar{u}$ satisfies \eqref{barusuper}-\eqref{barusuperboundary}-\eqref{barubottom}. Also, by definition $\bar{u}\geq 0$ on $\{x_0\}=\partial\Gamma\cap\partial\Omega$. Hence, by comparison, it follows that $\bar{u}\geq u$ in all of $\Gamma \times [0,\infty)$.

\emph{Step 3. Relating the comparison functions for small time.} We claim that for some $T^*>0$,
\begin{equation*}
	\bar{u}(x,t) \leq \bar{v}(x) \quad\text{for all } x \in\Gamma \text{ and } 0\leq t \leq T^*.
\end{equation*}

The proof is similar to that of the previous step. We establish first the comparison for $t=0$, the bottom of the domain. Again we consider $x = x_0 - s\nu$. Recalling \eqref{psiprime}, we now seek
\begin{equation}\label{ubarvbargradients}
	\left| \frac{\partial \bar{v}}{\partial \nu}(x_0)\right| = \varphi '(0) \geq C\gamma = \left|\frac{\partial \bar{u}}{\partial \nu}(x_0, 0)\right|.
\end{equation}

From previous computations,
\begin{equation*}
	\varphi '(0) = \mu^{-\beta} \rightarrow +\infty \quad \text{as } \mu \rightarrow 0.
\end{equation*}

On the other hand, $C$ depends on $\mu$ through \eqref{CversusGamma}. Since $\psi=\psi(s)$ is increasing inward, the minimum in \eqref{CversusGamma} is achieved at $s=\delta$. Clearly we can take $\delta < \mu=\eta$, and so a simple computation shows
\begin{equation*}
	C = o(\mu^{-\beta}) \quad\text{as } \mu \rightarrow 0.
\end{equation*}

Therefore, taking $\mu = \eta$ small enough eventually yields \eqref{ubarvbargradients}. We remark that this choice, which amounts to shrinking the domain $\Gamma$, does not affect the choices made for other constants.

As before, looking at the first order expansion gives $\bar{u}(x,0) \leq \bar{v}(x)$ for all $x$ as above, near $x_0$, say with $s<\delta'$. Moreover, in this case it is easier to extend the inequality to the directions which sweep $\Gamma \times \{t=0\}$, since both functions are radial and defined on the same annular domain.

To obtain 
\begin{equation}\label{ubarvbarouterbottom}
	\bar{u}(x,0) \leq \bar{v}(x) \quad\text{on } \partial\Gamma,
\end{equation}
there is no choice like \eqref{CversusGamma} available. However, we may restrict the comparison to $\Gamma_{\delta'} \times \{t=0\}$, where $\Gamma_{\delta'} := \Gamma\cap\{0 <s=|x-x_1| - \rho < \delta'\}$, $\delta'>0$ as above, so that \eqref{ubarvbarouterbottom} holds by comparing the first-order expansions, and more importantly, holds strictly. That is, $\bar{u}(x,0) < \bar{v} (x)$. We may now take $T^*$ small enough so that, for all $0\leq t \leq T^*$ and all $x\in \Omega$ such that $s = |x-x_1| - \rho = \delta'$,
\begin{equation*}
	u(x,t) \leq \bar{u}(x,t) = At + \bar{u}(x,0) \leq \bar{v}(x).
\end{equation*}

Thus we have proven that $\bar{v}$ solves
\begin{equation*}\left\{\begin{array}{ll}
	\bar{v}_t - \Mcalm(D^2 \bar{v}) \geq |D\bar{v}|^p &\quad\textrm{in } \Gamma_{\delta'} \times (0,T^*),\\
	\bar{v} \geq u &\quad\textrm{on } (\partial\Gamma_{\delta'} \cap \overline{\Omega}) \times [0,T^*],\\
	\bar{v} \geq 0 &\quad\text{on } (\partial\Gamma_{\delta'} \cap \partial\Omega) \times [0,T^*],\\
	\bar{v} \geq u_0 &\quad\text{in }\overline{\Gamma}_{\delta'},\end{array}\right.
\end{equation*}

where the boundary conditions (inequalities) are satisfied pointwise. Hence, by the comparison principle of \cite{barles2004generalized} (see also Remark \ref{comparisonuptoboundary}), $u(x,t) \leq \bar{v}(x)$ in all of $ \overline{\Gamma}_{\delta'} \times [0,T^*]$. In particular, this implies $u(x_0,t) \leq \bar{v}(x_0) = 0$, hence $u(x_0,t) = 0$ for all $0\leq t \leq T^*$. As $x_0\in\partial \Omega$ was arbitrary, this implies that the solution $u$ satisfies the boundary conditions in the classical sense on $\partial \Omega \times [0,T^*]$.\hfill\qed

\begin{remark}\label{rmkValuePLocal}
	The preceding computations do not require that $p>2$; only $p>1$ is explicitly used in \eqref{dom_pow}. We remark, however, that only the superquadratic case is of interest, since by the results of \cite{barles2004generalized}, in the subquadratic case the globally defined solution of \eqref{modeleq}-\eqref{boundarydata}-\eqref{initialdata} satisfies the boundary data in the classical sense.
\end{remark}

\begin{remark}\label{intSphereLocal}
	The proof of Theorem \ref{localexthm} uses only the uniform exterior sphere condition. Both interior and exterior sphere conditions are assumed to establish a connection between Theorems \ref{localexthm} and \ref{uniparabolthm}. See also Remark \ref{rmkextSphereUni}.
\end{remark}

\section{Technical results for the proof of nonexistence}\label{techsec}
In this section we gather a series of technical results and fundamental facts related, on one hand, to the process by which we arrive at an approximate equation (actually, an inequality) with the required properties, and on the other, to the eigenvalue problem for the Pucci operator. Furthermore, we introduce key concepts and notation that will be used in Sec. \ref{nonexistence}.

\subsection{Radial form}\label{radialform}
\begin{lemma}\label{radiallemma}
	Let $u\in C(\overline{B_1(0)})$ be the viscosity solution of
	\begin{equation}\label{eqinball}
		\left\{\begin{array}{ll}
			u_t - \Mcalm(D^2 u) - |Du|^p = 0 &\textrm{in } B_1(0) \times (0,T),\\
			u = 0 &\textrm{on } \partial B_1(0) \times [0,T],\\
			u(\cdot, 0) = u_0 &\textrm{in } \overline{B_1(0)},
		\end{array}\right.
	\end{equation}
	where $u_0$ is a radial function. Then $u$ is radial as well, that is, $u(x,t) = U(|x|,t)$ for some $U:[0,1] \times [0,T]\to \mathbb{R}$, and $U$ solves	
	\begin{equation}\label{radialeq}
		\left\{\begin{array}{ll}
			U_t - \theta(U'')U'' - \frac{n-1}{r}\theta(U')U' - |U'|^p = 0 &\textrm{in } (0,1)\times(0,T),\\
			U = 0 &\textrm{on } \{r=1\} \times [0,T],\\
			U(\cdot, 0) = u_0 &\textrm{in } B_1(0),
		\end{array}\right.
	\end{equation}
	in the viscosity sense, where $'$ denotes the radial derivative and 	
	\begin{equation*}
		\theta(s) = \left\{ \begin{array}{rl}
		\lambda, & \textrm{if}\ s>0, \\
		\Lambda, & \textrm{if}\ s\leq 0.
		\end{array}\right.
	\end{equation*}
\end{lemma}

\begin{remark}
	We note that, although the function $\theta$ above is discontinuous at $0$, equation \eqref{radialeq} depends continuously on the derivatives of $u$, since the function $s\mapsto \theta(s)s$ is continuous for all $s\in\mathbb{R}$.
\end{remark}

\textit{Proof:} The solution $u$ of \eqref{eqinball} is radial due to the uniqueness of solutions of the Cauchy-Dirichlet problem, the rotation-invariance of the equation and the fact that the initial data is radial. Hence, there exists a function $U:\mathbb{R}\rightarrow\mathbb{R}$ such that $u(x) = U(|x|)$ for all $x\in B_1(0)$. We will show that this function is a subsolution of \eqref{radialeq} by definition. 

Consider $\Phi((0,1)\times (0,T))\in C^2$ that touches $U$ from above at $(\hat{r},\hat{t})$, and  define $\phi(x,t) = \Phi(|x|,t)$. Then $\phi$ is $C^2$, radial and a valid test function for $u$ at any $(\hat{x},\hat{t})$ such that $|\hat{x}|=\hat{r}$, hence we can compute $\Mcalm(D^2\phi)$ as in \eqref{radialpucci}. Observing also that $|D\phi| = |\Phi '|$, we obtain exactly the equation in \eqref{radialeq}. The proof that $U$ is also a supersolution is analogous.\hfill\qed

\begin{remark}
	In what follows we will at times write simply $u(x)=u(r)$ for radial functions, as is standard. We avoided this notation in the last lemma for clarity.
\end{remark}

\subsection{Regularization}\label{regularization}

In this section is we apply the regularization procedure introduced in \cite{lasry1986remark} solution $u$ of \eqref{modeleq}. In this way we obtain an equation satisfied in the pointwise a.e.~sense, all of whose terms are integrable. Although the technique is applicable in greater generality, in practice we will only apply the regularization to solutions of \eqref{modeleq} when $\Omega = B_1(0)$ (see Remarks \ref{notationremark} and \ref{radialrmk}). For the sake of clarity, especially regarding notation, we recall some of the relevant definitions and properties, noting that we do not seek full generality in what follows.

\begin{definition}
	For $u\in C(\overline{\Omega} \times [0,T])$ and $\epsilon, \kappa > 0$, define
	\begin{equation}
	u_{\epsilon, \kappa}(x,t) = \inf_{(y,s)\in \Omega \times (0,T)} \left(u(y,s) + \frac{1}{2\epsilon} |x-y|^2 + \frac{1}{2\kappa}|t-s|^2\right),
	\end{equation}	
	\begin{equation}
	u^{\epsilon}(x,t) = \sup_{y\in \Omega} \left(u(y,t) - \frac{1}{2\epsilon} |x-y|^2\right).
	\end{equation}
\end{definition}

We may also define $u^{\epsilon, \kappa}$ and $u_{\epsilon}$ similarly. Note that we use just one index when the convolution is performed in the space variable only. In the following statement we collect a series of well-known facts regarding these operations which will be used shortly hereafter.
\begin{proposition}\label{supinfconvproperties}
	Assume $u\in C(\overline{\Omega} \times [0,T])$, and let $\epsilon, \kappa, \delta > 0$.
	\begin{enumerate}[(i)]
		\item Both operations preserve both pointwise upper and lower bounds, i.e., 
		\begin{align*}
		\inf u \leq u_{\epsilon, \kappa} \leq \sup u,\\
		\inf u \leq u^\epsilon \leq \sup u,
		\end{align*}
		where $\inf$ and $\sup$ are taken over $\Omega\times (0,T)$.
		\item\label{optimalpoints} Let $\epsilon^* = 2\sqrt{\epsilon\|u\|_\infty}$, $\kappa^* = 2\sqrt{\kappa\|u\|_\infty}$, $\Omega^{\epsilon^*} = \{ x \in \Omega\ | \ d(x,\partial \Omega)>\epsilon^*\}$. For all $(x,t)\in \Omega^{\epsilon^*} \times (\kappa^*, T-\kappa^*)$, there exist $(y,s)\in\Omega\times(0,T)$ such that
		\begin{equation*}
		u_{\epsilon, \kappa}(x,t) = u(y,s) + \frac{1}{2\epsilon} |x-y|^2 + \frac{1}{2\kappa}|t-s|^2.
		\end{equation*}
		In other words, the $\sup$ and $\inf$ in the definition of the convolutions are achieved, provided we are at a sufficient distance from the boundary.
		\item\label{supinflipschitz} Both $u_{\epsilon, \kappa}$ and $u^{\epsilon, \kappa}$ are Lipschitz continuous in $x$ with constant $\frac{K}{\sqrt{\epsilon}}$, where $K = 2\|u\|_\infty$. That is,
		\begin{equation*}
		\sup_{\substack{x,y\in\Omega\\t\in[0,T]}} \frac{|u(x,t) - u(y,t)|}{|x-y|} \leq \frac{K}{\sqrt{\epsilon}}.
		\end{equation*}
		Similarly, they are Lipschitz continuous in $t$ with constant $\frac{K}{\sqrt{\kappa}}$.
		
		\item $u^{\epsilon, \kappa}, u_{\epsilon, \kappa} \rightarrow u$ uniformly as $\epsilon, \kappa \rightarrow 0$, and similarly for $u^\epsilon$.
		
		\item\label{twicediff} $u^{\epsilon, \kappa}, u_{\epsilon, \kappa}$ are respectively semiconvex and semiconcave. In particular, they are twice differentiable a.e. That is, there are measurable functions $a:\Omega\times [0,T] \to \mathbb{R}$, ${q:\Omega\times [0,T] \to \mathbb{R}^n}$, ${M: \Omega\times [0,T] \to S(n)}$ such that
		\begin{align*}
		u^{\epsilon, \kappa}(y,s) ={}& u^{\epsilon, \kappa}(x,t) + a(x,t)(s-t) + \langle q(x,t),y-x \rangle\\
		&+ \langle M(x,t)(y-x),y-x \rangle + o(|y-x|^2 + |s-t|). 
		\end{align*}
		We will denote $a=(u^{\epsilon, \kappa})_t,\, q=Du^{\epsilon, \kappa},\, M= D^2 u^{\epsilon, \kappa}$ for simplicity. The same goes for $u_{\epsilon, \kappa}$.
		\item With the notation above,
		\begin{align*}
		D^2u_{\epsilon, \kappa} \leq \frac{1}{\epsilon}I \quad \textrm{and} \quad D^2u^{\epsilon, \kappa} \geq -\frac{1}{\epsilon}I  \quad a.e. \text{ in } \Omega\times [0,T].
		\end{align*}	
		\item\label{semigroupprop} $(u_{\epsilon, \kappa})_\delta = u_{\epsilon + \delta, \kappa}$.
		
		\item\label{doubleconvineq} $(u_{\epsilon + \delta, \kappa})^\delta \leq u_{\epsilon, \kappa}$.
	\end{enumerate}	
\end{proposition}

\begin{remark}
	The easier proofs follow more or less directly from the definitions (see e.g., \cite{dragoniintroduction}), while \eqref{semigroupprop} and \eqref{doubleconvineq} may be found in \cite{crandall1996equivalence}. Property \eqref{twicediff} uses the well-known theorems of Rademacher and Alexandrov on the differentiability of Lipschitz and convex functions, respectively; see \cite{evans2015measure} and the Appendix of \cite{crandall1992user}.
\end{remark}

The time-independent version of the following result appears as Lemma 4.2 in \cite{silvestre2015viscosity}. We say that $F$ is proper if for all $(X,\xi)\in S(n)\times \mathbb{R}^n$, $r,s\in \mathbb{R}$, if $r\leq s$ then $F(X,\xi,r) \leq F(X,\xi,s)$.
\begin{lemma}\label{infconvpreserves}
	Let $u$ be a viscosity supersolution of $u_t + F(D^2u, Du, u) = 0$ in $\Omega \times (0,T)$, where $F$ is proper. Then, using the notation of Proposition \ref{supinfconvproperties}, $u_{\epsilon, \kappa}$ is a viscosity supersolution of $u_t + F(D^2u, Du, u) = 0$ in $\Omega^{\epsilon^*}\times (\kappa^*, T - \kappa^*)$.
\end{lemma}

\textit{Proof:} Let $\varphi = \varphi(x,t)$ be a $C^2$ function that touches $u_{\epsilon, \kappa}$ from below at $(\hat{x}, \hat{t}) \in \Omega^{\epsilon^*}\times (\kappa^*, T - \kappa^*)$, that is, $\varphi(\hat{x}, \hat{t}) = u_{\epsilon, \kappa}(\hat{x}, \hat{t})$ and for $|x-\hat{x}| + |t-\hat{t}| < \delta$ and sufficiently small $\delta>0$,
\begin{equation}\label{xtClose}
\varphi(x,t) \leq  u_{\epsilon, \kappa} (x,t).
\end{equation}

By Proposition \ref{supinfconvproperties}, \eqref{optimalpoints} there exist $(\hat{y}, \hat{s}) \in \Omega \times (0,T)$ such that
\begin{equation*}
u_{\epsilon,\kappa} (\hat{x}, \hat{t}) = u(\hat{y}, \hat{s}) + \frac{1}{2\epsilon}|\hat{x} - \hat{y}|^2 + \frac{1}{2\kappa}|\hat{t} - \hat{s}|^2,
\end{equation*}
with $(\hat{y}, \hat{s})\to (\hat{x}, \hat{t})$ as $\epsilon,\kappa\to 0$. Hence, for sufficiently small $\epsilon,\kappa$, $(\hat{y}, \hat{s})$ remains close to $(x,t)$ as in \eqref{xtClose}. Therefore,
\begin{equation*}
\varphi(x,t) \leq u_{\epsilon, \kappa} (x,t) \leq u( x + (\hat{y} - \hat{x}), t + (\hat{s} - \hat{t})) + \frac{1}{2\epsilon}|\hat{x} - \hat{y}|^2 + \frac{1}{2\kappa}|\hat{t} - \hat{s}|^2.
\end{equation*}
Evaluating this expression now at $(x + (\hat{x} - \hat{y}), t + (\hat{t} - \hat{s}))$, we have that
\begin{equation*}
\tilde{\varphi} (x,t) := \varphi (x + (\hat{x} - \hat{y}), t + (\hat{t} - \hat{s}) ) - \frac{1}{2\epsilon}|\hat{x} - \hat{y}|^2 - \frac{1}{2\kappa}|\hat{t} - \hat{s}|^2 \leq u(x,t).
\end{equation*}

We also have, from the choice of $(\hat{y}, \hat{s})$, that $\tilde{\varphi} (\hat{y}, \hat{s}) = u (\hat{y}, \hat{s})$. Hence $\tilde{\varphi}$ is a valid test function for $u$. Since
\begin{equation*}
\begin{array}{l}
D^2\tilde{\varphi}(\hat{y}, \hat{s}) = D^2\varphi(\hat{x}, \hat{t}), \quad D\tilde{\varphi}(\hat{y}, \hat{s}) = D\varphi(\hat{x}, \hat{t}),\\
\tilde{\varphi}_t(\hat{y}, \hat{s}) = \varphi_t(\hat{x}, \hat{t}), \quad \tilde{\varphi}(\hat{y}, \hat{s}) \leq \varphi(\hat{x}, \hat{t}),		
\end{array}
\end{equation*}
by the properness of $F$,
\begin{equation*}
\begin{array}{l}
\varphi_t(\hat{x}, \hat{t}) + F(D^2\varphi(\hat{x}, \hat{t}), D\varphi(\hat{x}, \hat{t}), \varphi(\hat{x}, \hat{t})) \geq  \\ 
\quad \tilde{\varphi}_t(\hat{y}, \hat{s}) + F(D^2\tilde{\varphi}(\hat{y}, \hat{s}), D\tilde{\varphi}(\hat{y}, \hat{s}), \tilde{\varphi}(\hat{y}, \hat{s})) \geq 0.
\end{array}
\end{equation*}
Hence, $u_{\epsilon, \kappa}$ is a supersolution. That it is a subsolution is proved similarly.\hfill\qed

The following Proposition is an adaptation of Proposition 4.6 from \cite{crandall1996equivalence}, which is developed in a slightly different context.

\begin{proposition}\label{appeqlemma}
	Let $\Omega'\!\subset\subset\! \Omega,\ 0<t_0<t_1<T$ and $u$ be a bounded viscosity supersolution of \eqref{modeleq} in $\Omega \times (0,T)$. Then, there exist constants $\epsilon, \delta, \kappa > 0$ such that the regularized function $w = (u_{\epsilon + \delta, \kappa})^\delta$ satisfies
	\begin{equation}\label{appeq}
	w_t - \Mcalm(D^2w) \geq |Dw|^p \quad \text{a.e. in}\ \Omega' \times (t_0, t_1).
	\end{equation}	
	In particular, $w$ is a so-called $L^\infty$-strong supersolution of \eqref{appeq}.
\end{proposition}

\textit{Proof:} We apply Lemma \ref{infconvpreserves} to \eqref{modeleq} to find that, for sufficiently small $\epsilon$ and $\kappa$, $u_{\epsilon, \kappa}$, depending on $\|u\|_\infty(\Omega\times(0,T))$ (see Proposition \ref{supinfconvproperties}, \eqref{optimalpoints}) is a viscosity supersolution of \eqref{modeleq} in $\Omega' \times (t_0, t_1)$. Observe that the Lemma applies since there is no $x$ dependence. The regularized function $w$ defined above is both semiconvex and semiconcave in $x$, as well as Lipschitz-continuous in $t$. Hence it is twice differentiable a.e.~in $\Omega' \times (\tau, T - \tau)$, in the sense of having a second order ``parabolic" Taylor expansion (as in Proposition \ref{supinfconvproperties}, \eqref{twicediff}).

Let $(\hat{x},\hat{t})$ be any such point of differentiability. As in Proposition \ref{supinfconvproperties}, \eqref{doubleconvineq}, we have that $w \leq u_{\epsilon, \kappa}$. Suppose that $w(\hat{x},\hat{t}) = u_{\epsilon, \kappa}(\hat{x},\hat{t})$. For $(x,t)$ in a neighborhood of $(\hat{x},\hat{t})$, we then have
\begin{align*}
u_{\epsilon, \kappa}(x,t) \geq w(x,t) ={}& w(\hat{x},\hat{t}) + w_t(\hat{x},\hat{t})(t-\hat{t}) + \langle Dw(\hat{x},\hat{t}), x - \hat{x}\rangle\\
&+ \langle D^2w(\hat{x},\hat{t}), x - \hat{x}\rangle  + o(|x-\hat{x}|^2 + |t - \hat{t}|)\\
={}&u_{\epsilon, \kappa} (\hat{x},\hat{t}) + w_t(\hat{x},\hat{t})(t-\hat{t}) + \langle Dw(\hat{x},\hat{t}), x - \hat{x}\rangle\\
&+ \langle D^2w(\hat{x},\hat{t}), x - \hat{x}\rangle  + o(|x-\hat{x}|^2 + |t - \hat{t}|),
\end{align*}
which implies that $(w_t(\hat{x}, \hat{t}), Dw(\hat{x}, \hat{t}), D^2w(\hat{x}, \hat{t})) \in \mathcal{P}^{2,-}u_{\epsilon, \kappa}(\hat{x}, \hat{t})$, the parabolic subjet at $(\hat{x}, \hat{t})$ (see, e.g., \cite{crandall1992user}). Since $u_{\epsilon, \kappa}$ is a viscosity supersolution, this gives
\begin{equation*}
w_t(\hat{x},\hat{t}) - \Mcalm(D^2w(\hat{x},\hat{t})) - |Dw(\hat{x},\hat{t})|^p \geq 0.
\end{equation*}

Assume now that $w(\hat{x},\hat{t}) < u_{\epsilon, \kappa}(\hat{x},\hat{t})$. In this case, by Proposition 4.4 in \cite{crandall1996equivalence}, $D^2w(\hat{x},\hat{t})$ has an eigenvalue equal to $-\frac{1}{\delta}$. On the other hand, by Proposition 4.5 in \cite{crandall1996equivalence}, $w$ is $\frac{1}{2\epsilon}$-semiconvex, so the remaining eigenvalues are bounded by above by $\frac{1}{\epsilon}$. Recalling also the gradient bounds which come from the Lipschitz continuity of $w$ with respect to both $x$ and $t$, as in Proposition \ref{supinfconvproperties}, \eqref{supinflipschitz}, we obtain
\begin{equation*}
w_t(\hat{x},\hat{t}) - \Mcalm(D^2w(\hat{x},\hat{t})) - |Dw(\hat{x},\hat{t})|^p \geq - \frac{K}{\kappa^\frac{1}{2}} + \lambda\frac{1}{\delta} - (n-1)\Lambda\frac{1}{\epsilon} - \frac{K^p}{\epsilon^\frac{p}{2}}.
\end{equation*}	
By taking $\delta = o(\epsilon^\frac{p}{2})$ and $\epsilon$ sufficiently small, the right-hand side of the above inequality becomes nonnegative. Hence, $w$ is a supersolution.\hfill\qed

\begin{remark}\label{notationremark} For the proof of Theorem \ref{mainonball} we apply Proposition \ref{appeqlemma} in the case where $\Omega=(0,1)$, and regularization is applied to $U=U(r)$, the radial part of the solution $u$ of \eqref{modeleq} in $B_1(0)\times[0,T]$. The spatial regularization will be performed with respect to the radial variable. To alleviate the notation of Section \ref{nonexistence}, we briefly switch to using $\tilde{\epsilon}$ and $\tilde{\delta}$ for the spatial regularization parameters. Precisely, this gives
	\begin{align*}
	& w(r,t) = (U_{\tilde{\epsilon} + \tilde{\delta}, \kappa})^{\tilde{\delta}} (r,t) \\
	& = \sup_{r_1\in(0,1)} \inf_{\substack{r_2\in(0,1)\\s\in(0,T)}} \left( U(r_2,s) + \frac{1}{2(\tilde{\epsilon} + \tilde{\delta})} |r_2-r_1|^2 + \frac{1}{2\kappa}|t-s|^2 - \frac{1}{2\tilde{\delta}}|r-r_1|^2 \right).
	\end{align*}
	Note also that from the proof of Proposition \ref{appeqlemma}, we choose $\tilde{\delta}=\tilde{\delta}({\tilde{\epsilon}})$, with $\tilde{\delta}\rightarrow0$ as $\tilde{\epsilon}\rightarrow 0$, so we need only choose suitable $\tilde{\epsilon}>0$ in the regularization.
	
	Since the viscosity solution of \eqref{modeleq} is uniformly bounded (see Remark \ref{uniformboundu}), Proposition \ref{appeqlemma} provides a supersolution to \eqref{modeleq} on a domain which arbitrarily approaches $(0,1)\times(0,T)$. That is, we can have $w$ satisfy \eqref{appeq} in $(\epsilon, 1 - \epsilon)\times(t_0,t_1)$ for arbitrarily small $\epsilon$ and $t_0$, and $t_1$ close to $T$, provided we choose small enough regularization parameters, depending on $\|u_0\|_\infty$. For this reason, in the proof of Theorem \ref{mainonball} given in the following section, we require certain estimates as $\epsilon,\ t_0\to 0$. This use of $\epsilon$ is maintained from the following subsection onwards, throughout Section \ref{nonexistence}, where additionally $\delta$ is used as a different cut-off parameter.
	
	The use of $\epsilon$ and $\delta$ as regularization parameters is only briefly revisited in Subsection \ref{subsLOBCextens}, where some comments are made regarding the adaptation of Proposition \ref{appeqlemma} to an equation with more general nonlinearities.
\end{remark}

\begin{remark}\label{radialrmk}
	We  obtain the inequality in divergence form as follows. By combining Proposition \ref{appeqlemma}, Lemma \ref{radiallemma}, and the considerations of Remark \ref{notationremark}, we have $w$ satisfies
	\begin{equation}\label{radialineq}
	w_t - \theta(w'')w'' - \frac{n-1}{r}\theta(w')w' - |w'|^p \geq 0 \quad\textrm{for a.e. } r\in(\epsilon,1-\epsilon), t\in (t_0,t_1),
	\end{equation}
	for arbitrarily small $\epsilon, t_0>0$ and $t_1$ arbitrarily close to $T$. Define
	\begin{equation*}
	\hat{n} = \frac{\theta(w')}{\theta(w'')} (n - 1) + 1,\quad \rho(r) = e^{\int_{1-\epsilon}^{r}\frac{\hat{n}-1}{s} \,ds},\quad\text{and} \quad\tilde{\rho}(r)= \frac{\rho(r)}{\theta(w'')}.
	\end{equation*}
%	and
%	\begin{equation*}
%	\tilde{\rho}(r)= \frac{\rho(r)}{\theta(w'')}.
%	\end{equation*}
	Note that $\rho$ is the indefinite integral of a measurable function, and is therefore absolutely continuous. In particular, this implies that it is differentiable a.e. 
	
	Multiplying \eqref{radialineq} by $\tilde{\rho}$, we obtain, for the second-order terms,
	\begin{equation*}
	\tilde{\rho}\left(\theta(w'')w'' + \frac{n-1}{r}\theta(w')w'\right) = (\rho w')'.
	\end{equation*}
	Hence,
	\begin{equation}\label{divformeq}
	\tilde{\rho} w_t \geq (\rho w')' + 	\tilde{\rho}|w'|^p \quad \textrm{for a.e. } r\in(\epsilon,1-\epsilon), t\in (t_0,t_1).
	\end{equation}
\end{remark}

\begin{remark}\label{underweightrmk}
	The functions $\rho$ and $\tilde{\rho}$ depend on the regularization parameters both explicitly and through the solution of the approximate equation $w$, but we omit these dependencies for simplicity of notation.
	
	We now provide a couple of bounds which will be useful later. As
	\begin{equation*}
		\hat{n} - 1 = \frac{\theta(w')}{\theta(w'')} (n - 1)  \leq \frac{\Lambda}{\lambda}(n-1),
	\end{equation*}
	we have for $\epsilon\in(0,\frac{1}{2})$ and all $r\in(\epsilon, 1 - \epsilon)$, that
	\begin{align}
		\tilde{\rho}(r) ={}& \frac{1}{\theta(w'')}e^{\int_{1-\epsilon}^{r}\frac{\hat{n}-1}{s} \,ds} \geq \frac{1}{\Lambda}\left(\frac{r}{1-\epsilon}\right)^{\frac{\Lambda}{\lambda}(n-1)} \nonumber\\
		\geq{}& \frac{1}{\Lambda}\left(\frac{r}{2}\right)^{\frac{\Lambda}{\lambda}(n-1)}:=\hat{\rho}(r).\label{underweight}
	\end{align}
	Note that $\hat{\rho}$ no longer depends on the regularization parameters and is defined for all ${r\in (0,1)}$. This is the function which appears in the statement of Theorem \ref{mainonball}. 
	
	On the other hand, since $\frac{r}{1-\epsilon}\leq 1$ for all $r\in(\epsilon, 1 - \epsilon)$, we similarly obtain
	\begin{equation}\label{overweight}
	\tilde{\rho}(r) \leq \frac{1}{\lambda}.
	\end{equation}
	We can also explicitly compute
	\begin{equation}\label{rhoendpoints}
		\rho(1-\epsilon) = 1\quad\text{and}\quad0\leq \rho(\epsilon) \leq \left(\frac{\epsilon}{1-\epsilon}\right)^\frac{\lambda(n-1)}{\Lambda} \rightarrow 0 \quad\textrm{as }\epsilon\rightarrow 0.
	\end{equation}
%	\begin{equation}\label{rho1-eps}
%	\rho(1-\epsilon) = 1
%	\end{equation}
%	and
%	\begin{equation}\label{rhoeps}
%	0\leq \rho(\epsilon) \leq \left(\frac{\epsilon}{1-\epsilon}\right)^\frac{\lambda(n-1)}{\Lambda} \rightarrow 0 \quad\textrm{as }\epsilon\rightarrow 0^+.
%	\end{equation}
\end{remark}

\subsection{Eigenvalue problem for the Pucci extremal operator}

The proof of Theorem \ref{mainonball} involves the solution to the Dirichlet eigenvalue problem for the extremal operator $-\Mcalm$ in annular domains approximating the punctured ball $B_1(0)\backslash\{0\}$.

More precisely, let ${A_{\epsilon} = B_{1-\epsilon}(0)\backslash\overline{B_{\epsilon}(0)}}$, and consider
\begin{equation}\label{puccieigenvalue}
\left\{\begin{array}{ll}
-\Mcalm(D^2\varphi) = \lambda\varphi & \quad \textrm{in } A_{\epsilon},\\
\varphi = 0 & \quad \textrm{on } \partial A_{\epsilon}.
\end{array}\right.
\end{equation}
Here the boundary condition is satisfied in the classical sense. Note that $A_{\epsilon}$ corresponds to the spatial domain where \eqref{divformeq} is satisfied.

By Proposition 1.1 in \cite{esteban2005nonlinear}, there exists a solution pair $(\lambda_1^\epsilon, \varphi_1^\epsilon)$ of \eqref{puccieigenvalue} with ${\lambda_1^\epsilon > 0}$, ${\varphi_1^\epsilon\in C^2(A_{\epsilon}) \cap C(\overline{A_{\epsilon}})}$ and $\varphi_1^\epsilon > 0$ in $A_{\epsilon}$, where $\varphi_1^\epsilon$ is unique up to a positive constant. We normalize this solution so that $\varphi_1^\epsilon(\frac{1}{2}) = 1$, for reasons that will become apparent later. Our notation indicates that both $\lambda_1^\epsilon$ and $\varphi_1^\epsilon$ depend on the parameters of the spatial regularization through the domain $A_{\epsilon}$. 

We employ the following lemma to state our main theorem without reference to these regularization parameters.

\begin{lemma}\label{unilowerboundforphi}
	Let $K\subset(0,1)$ be a closed interval such that $[\nicefrac{1}{4}, \nicefrac{3}{4}]\subset K$. There exists a function $\hat{\varphi}\in C(K)$, such that $\hat{\varphi}(r)>0$ for all $r\in K$ and $\varphi_1^\epsilon \rightarrow \hat{\varphi}$ uniformly over $K$, up to a subsequence.
\end{lemma}

\begin{proof} 
	In general, if we denote by $\lambda_1(\Omega)$ the corresponding \emph{principal half-eigen\-value} (i.e., solution of \eqref{puccieigenvalue}) in $\Omega$, we have that $\lambda_1(\Omega') \leq \lambda_1(\Omega)$ if $\Omega \subset \Omega'$; see Proposition 1.1 (\textit{iii}) in \cite{esteban2005nonlinear}. We therefore have the monotonicity $\lambda_1^{\epsilon'} \leq \lambda_1^\epsilon$ if $\epsilon' \leq \epsilon$. For the same reason, $\lambda_1(B_1(0)) \leq \lambda_1^\epsilon$ for all $\epsilon>0$. Hence, $\lambda_1^\epsilon \rightarrow \hat{\lambda}$ as $\epsilon \rightarrow 0$, for some $\hat{\lambda} > 0$. Note also that $\lambda_1^\epsilon\leq \lambda((\nicefrac{1}{4}, \nicefrac{3}{4}))$ for all $\epsilon<\nicefrac{1}{4}$.
	
	Consider now a closed interval $K'\supset K$. By Harnack's inequality (see Theorem 3.6 in \cite{quaas2008principal}), for all $\epsilon > 0$ small enough such that $K'\subset (\epsilon, 1-\epsilon)$, we have
	\begin{equation}\label{upboundphi}
		\sup_{K'} \varphi_1^\epsilon \leq \sup_{(\epsilon, 1-\epsilon)} \varphi_1^\epsilon \leq C \inf_{(\epsilon, 1-\epsilon)} \varphi_1^\epsilon \leq C \inf_{K'} \varphi_1^\epsilon \leq C,
	\end{equation}
	where we used $\varphi_1^\epsilon(\frac{1}{2}) = 1$ and $\frac{1}{2} \in K'$ for the last inequality. The Harnack constant $C$ above depends only on $n, \lambda, \Lambda, \lambda_1^\epsilon$ and $\mathrm{dist}(K',\partial(0,1))$. Since $\lambda_1^\epsilon$ is uniformly bounded for $\epsilon<\nicefrac{1}{4}$, $C$ is independent of $\epsilon$ as well.
	
	It follows that the functions $\varphi_1^\epsilon$ are uniformly bounded, and therefore satisfy a family of ODEs with uniformly bounded right-hand sides. More precisely, using \eqref{radialpucci} once more, we have
	\begin{align}\label{eigenODEs}
		\Mcalm(D^2\varphi_1^\epsilon) = \theta[(\varphi_1^\epsilon)''](\varphi_1^\epsilon)'' + \theta[(\varphi_1^\epsilon)'](\varphi_1^\epsilon)'\frac{n-1}{r} = -\lambda_1^\epsilon \varphi_1^\epsilon \quad\text{in } \mathrm{int}(K'),
	\end{align}
	with $\|\lambda_1^\epsilon \varphi_1^\epsilon\|_\infty \leq C\max\{\hat{\lambda}, \lambda((\nicefrac{1}{4}, \nicefrac{1}{4}))\}:=\overline{C}$, where $C$ is the Harnack constant above. We proceed with a compactness argument, following \cite{esteban2010eigenvalues}.
	
	Let $\gamma_\epsilon:=\|\varphi_1^\epsilon\|_{C^1(K')}$ and define $\tilde{\varphi^\epsilon} = \nicefrac{\varphi_1^\epsilon}{\gamma_\epsilon}$. Using that $\Mcalm$ is positive homogeneous, we have that $\tilde{\varphi^\epsilon}$ is also a solution of \eqref{eigenODEs}, and since $\|\tilde{\varphi^\epsilon}\|_{C^1(K')}=1$ for all $\epsilon<\nicefrac{1}{4}$, this implies $(\tilde{\varphi^\epsilon})''$ is uniformly bounded as well. 
	
	By compactness, this implies that, up to a subsequence, $\tilde{\varphi^\epsilon}\to\tilde{\varphi}$ uniformly on $K'$ for some $\tilde{\varphi}\in C^2(\mathrm{int}(K'))\cap C^1(K')$. Recalling that $\lambda_1^\epsilon \rightarrow \hat{\lambda}$, we pass to the limit to find
	\begin{equation}\label{eigenlimiteq}
		\Mcalm(D^2\tilde{\varphi}) = -\hat{\lambda}\tilde{\varphi} \quad\text{in } \mathrm{int}(K').
	\end{equation}
	Note $\tilde{\varphi} \not\equiv 0$ since it is the limit of $\tilde{\varphi^\epsilon}$ and $\|\tilde{\varphi^\epsilon}\|_{C^1(K')}=1$ for all $\epsilon<\nicefrac{1}{4}$.
	
	Assume that $\gamma_\epsilon$ becomes unbounded as $\epsilon\to 0$. As before, using the homogeneity of $\Mcalm$, we see that
	\begin{equation*}
		\Mcalm(D^2\tilde{\varphi^\epsilon}) = -\frac{\lambda_1^\epsilon\varphi_1^\epsilon}{\gamma_\epsilon} \quad\text{in } \mathrm{int}(K'),
	\end{equation*}
	and $\|\frac{\lambda_1^\epsilon\varphi_1^\epsilon}{\gamma_\epsilon}\|_\infty \leq \nicefrac{\overline{C}}{\gamma_\epsilon} \to 0$ as $\epsilon\to 0$. Passing to the limit, this gives $\Mcalm(D^2\tilde{\varphi}) = 0$, in contradiction with \eqref{eigenlimiteq}. It follows that $\gamma_\epsilon = \|\varphi_1^\epsilon\|_{C^1(K')}$ is uniformly bounded in $\epsilon$, hence we can pass to the limit as before (\emph{without} the normalization $\tilde{\varphi^\epsilon}$), to find that $\varphi_1^\epsilon\to\hat{\varphi}$ uniformly in $K'$ for some $\hat{\varphi}\in C(K')$ which is a solution of \eqref{eigenlimiteq}.
	
	From the uniform convergence and $\varphi_1^\epsilon \geq 0$ for all $\epsilon>0$, we conclude $\hat{\varphi} \geq 0$, hence the strong maximal principle applies (see Lemma 3.4 in \cite{gilbarg2015elliptic}). Recalling also that $\varphi_1^\epsilon(\frac{1}{2}) = 1$, we have $\hat{\varphi} \not\equiv 0$. Combining these facts, we conclude that $\hat{\varphi} > 0$ in $\mathrm{int}(K')\supset K$.\end{proof}

\begin{remark}\label{unilowerboundforphi-rmk}
	Since $K$ is a closed interval, $\hat{\varphi}$ is bounded by below on $K$ by a positive constant which does not depend on $\epsilon$.
\end{remark}

The proof of Theorem \ref{mainonball} requires two additional lemmas.
\begin{lemma}\label{eigenlemma}
	Let $\varphi_1^\epsilon$ be the solution of \eqref{puccieigenvalue}, as defined above. Then, for any $0<\alpha<1$, there exists a positive constant $C>0$ such that
	\begin{equation}\label{eigenint}
		\int_{\epsilon}^{1 - \epsilon} (\varphi_1^\epsilon)^{-\alpha} \tilde{\rho} \,dr < C.
	\end{equation}
	Furthermore, $C$ may be taken uniformly for $\epsilon\in(0,\frac{1}{4})$.
\end{lemma}

\begin{proof}
	By Remark \ref{unilowerboundforphi-rmk}, it is possible to bound $\varphi_1^\epsilon$ by below by a positive constant uniformly for small  $\epsilon$ over a closed interval $K\subset (0,1)$, to be chosen later. Hence, to obtain \eqref{eigenint} it is sufficient to bound the integral near the endpoints $\epsilon$ and $1-\epsilon$.
	
	We now proceed as in the proof of Hopf's lemma to obtain a uniform lower bound for $(\varphi_1^\epsilon)'(\epsilon)$. Let $\beta>0$ and
	\begin{equation*}
		v(r) = \frac{e^{-\beta(\frac{1}{2} - r)^2} - e^{-\beta(\frac{1}{2} - \epsilon)^2}}{1 - e^{-\beta(\frac{1}{2} - \epsilon)^2}} \quad \textrm{ for }\quad \epsilon<r<\frac{1}{2}.
	\end{equation*}
	We verify that $v\geq 0, v(\epsilon)=0, v(\frac{1}{2})=1$, and compute
	\begin{align}
		v'(r) ={}& \frac{2\beta(\frac{1}{2} - r)\,e^{-\beta(\frac{1}{2} - r)^2}}{1 - e^{-\beta(\frac{1}{2} - \epsilon)^2}} \geq 0,\label{unihopf1}\\[10pt]
		v''(r) ={}& \frac{\left(4\beta^2(\frac{1}{2} - r)^2 - 2\beta\right)e^{-\beta(\frac{1}{2} - r)^2}}{1 - e^{-\beta(\frac{1}{2} - \epsilon)^2}} \label{unihopf2}\geq 0,
	\end{align}
	where the inequality in \eqref{unihopf2} follows from taking a sufficiently large $\beta>0$. 
	
	We abuse notation slightly and define $v(x) = v(|x|)$ in $B_\frac{1}{2}(0)\backslash\overline{B_\epsilon(0)}$. By the previous computation, using also that $\lambda_1^\epsilon > 0, v\geq 0$, we have
	\begin{equation*}
		\Mcalm(D^2v) = \lambda v'' + \lambda\frac{n-1}{r}v' \geq 0 \geq -\lambda_1^\epsilon v.
	\end{equation*}
	Hence $v$ is a subsolution of \eqref{puccieigenvalue}. Since $v(\epsilon)=\varphi_1^\epsilon(\epsilon), v(\frac{1}{2})=\varphi_1^\epsilon(\frac{1}{2})$, by comparison we have $v(r)\leq\varphi_1^\epsilon(r)$ for all $\epsilon<r<\frac{1}{2}$ (see, for example, Appendix A in \cite{armstrong2009principal}). Recalling \eqref{unihopf1}, for all $0<\epsilon<\frac{1}{4}$ this gives 
	\begin{equation}\label{boundphi'below}
		(\varphi_1^\epsilon)'(\epsilon) \geq v'(\epsilon) = \frac{2\beta(\frac{1}{2} - \epsilon)\,e^{-\beta(\frac{1}{2} - \epsilon)^2}}{1 - e^{-\beta(\frac{1}{2} - \epsilon)^2}} \geq \frac{\beta e^{-\frac{\beta}{4}}}{1 - e^{-\frac{\beta}{4}}}=:C.
	\end{equation}
	Note that the last constant does not depend on $\epsilon$.
	
	By looking at the first order expansion of $\varphi_1^\epsilon$ at $\epsilon$,
	\begin{equation*}
		\varphi_1^\epsilon(r) = (\varphi_1^\epsilon)'(\epsilon)(r-\epsilon) + o(|r-\epsilon|),
	\end{equation*}
	we have that there exists a $\delta>0$ such that
	\begin{equation}\label{deltacutoff}
		\varphi_1^\epsilon(r) > \frac{(\varphi_1^\epsilon)'(\epsilon)}{2} (r-\epsilon) \quad \textrm{for all } \quad \epsilon<r<\epsilon + \delta.
	\end{equation}
	In the above series expansion, the constant $\delta>0$ depends only $(\varphi_1^\epsilon)'(\epsilon)$. In view of \eqref{boundphi'below}, we need only bound $(\varphi_1^\epsilon)'(\epsilon)$ by above, independently of $\epsilon$. For this we use a barrier type argument, taking advantage of some of the computations from Sec. \ref{localex}. Define
	\begin{equation*}
		\psi(r) = A\left(1- e^{-\gamma(r-\epsilon)}\right),
	\end{equation*}
	where $A, \gamma>0$ are to be chosen. We have $\psi(\epsilon)=0$, and for an appropriate choice of $A$,
	\begin{equation*}
		\psi(\nicefrac{1}{2}) = A\left(1 - e^{-\gamma(\nicefrac{1}{2} - \epsilon)}\right) \geq A(1 - e^{-\nicefrac{\gamma}{4}}) > 1,
	\end{equation*}
	again using $\epsilon<\nicefrac{1}{4}$. Computing as in \eqref{puccipsi}, and using once more that $\lambda_1^\epsilon \leq \lambda((\nicefrac{1}{4}, \nicefrac{3}{4}))$, we can choose $\gamma$ large enough so that
	\begin{align*}
		&-\Mcalm(D^2\psi) - \lambda_1^\epsilon\psi \geq -\Mcalm(D^2\psi) - \lambda((\nicefrac{1}{4}, \nicefrac{3}{4}))\psi\\
		&\qquad= A\left(\Lambda\gamma^2 - \lambda\gamma\frac{n-1}{r+\epsilon} + \lambda((\nicefrac{1}{4}, \nicefrac{3}{4})) \right)e^{-\gamma(r-\epsilon)} - A\lambda((\nicefrac{1}{4}, \nicefrac{3}{4}))\\
		&\qquad\geq A\left(\Lambda\gamma^2 - \lambda\gamma\frac{n-1}{r}\right)e^{-\gamma r} - A\lambda((\nicefrac{1}{4}, \nicefrac{3}{4})) \geq 0.
	\end{align*}
	Thus, by comparison, $\psi \geq \varphi_1^\epsilon$ in $[\epsilon, 1-\epsilon]$, for all $\epsilon<\nicefrac{1}{4}$. Hence, 
	\begin{equation}\label{boundphi'eps}
		(\varphi_1^\epsilon)'(\epsilon) \leq \psi'(\epsilon) = A\gamma,
	\end{equation}
	and from this we conclude that $\delta$ in \eqref{deltacutoff} does not depend on $\epsilon$.
	
	We then estimate, for $\epsilon<r<\epsilon+\delta$,
	\begin{align*}
		\left(\varphi_1^\epsilon(r)\right)^{-\alpha} <{}& \left(\frac{(\varphi_1^\epsilon)'(\epsilon)}{2}\right)^{-\alpha}(r-\epsilon)^{-\alpha}\\
		\leq{}& (\nicefrac{C}{2})^{-\alpha}(r-\epsilon)^{-\alpha},
	\end{align*}
	where we have used \eqref{boundphi'below} for the second inequality. Recall the bound $\tilde{\rho}(r)\leq \frac{1}{\lambda}$ given by \eqref{overweight}. Then,
	\begin{align*}
		\int_{\epsilon}^{\epsilon + \delta} (\varphi_1^\epsilon(r))^{-\alpha} \tilde{\rho} \,dr \leq{}&  (\nicefrac{C}{2})^{-\alpha}\|\tilde{\rho}\|_\infty \int_{\epsilon}^{\epsilon + \delta} (r-\epsilon)^{-\alpha} \,dr \\
		\leq{}&  (\nicefrac{C}{2})^{-\alpha}\,\frac{1}{\lambda} \int_{0}^{\delta} r^{-\alpha} \,dr < \tilde{C}.
	\end{align*}
	A similar bound can be obtained over $1-\epsilon-\delta < r < 1 - \epsilon$. We may then choose $K=[\delta, 1-\delta]$ with $\delta$ as above and recall that $\varphi_1^\epsilon$ are uniformly bounded by below on $K$. Also note that for all $\epsilon>0$,
	\begin{equation}
		K=[\delta,1-\delta] \supset (\epsilon+\delta, 1-\epsilon-\delta), 
	\end{equation}
	hence, we may combine the bounds near the endpoints with the lower bound on $K$ to obtain \eqref{eigenint}.
\end{proof}

\begin{remark}\label{choiceofK}
	The interval $K = [\delta, 1-\delta]$ is the one that appears in the statement of Theorem \ref{mainonball}. As claimed, $\delta$ depends only on $\lambda, \Lambda, n$.
\end{remark}

\begin{lemma}\label{weighedpoincare}
	Consider $\rho: [\epsilon,1-\epsilon]\to \mathbb{R}$ as defined in Remark \ref{radialrmk}. Let ${v:[\epsilon, 1-\epsilon]\to \mathbb{R}}$ be once differentiable such that $v(1-\epsilon) = 0$. Then,
	\begin{equation}
	\int_\epsilon^{1-\epsilon} |v|\, \rho\,dx \leq \int_\epsilon^{1-\epsilon} |v'|\, \rho\,dx.
	\end{equation}
\end{lemma}

In the proof of the above inequality we employ the following result from \cite{steinerberger2015sharp}:
\begin{theorem}\label{steinerbergpoinc}
	Let $\nu: [0,1] \to \mathbb{R}$ be a non-negative, non-vanishing, continuous weight on the closed unit interval. Let $f: [0,1] \to \mathbb{R}$ be once differentiable and satisfy $f(0)=0$. Then,
	\begin{equation}\label{steinerbergerineq}
	\int_{0}^{1} |f(x)| \nu(x) \,dx \leq \left(\max_{0\leq x \leq 1}\frac{1}{\nu(x)}\int_{x}^{1}\nu(z) \,dz\right) \int_{0}^{1} |f'(x)|\nu(x)\,dx,
	\end{equation}
	and the constant is sharp.
\end{theorem}

\noindent\textit{Proof of Lemma \ref{weighedpoincare}:} Let $v$ be as in the statement of the Lemma, and define $g:[0,1]\rightarrow \mathbb{R},\ g(r)=(1-2\epsilon)r + \epsilon$. Note that $g$ is an affine change of variables sending $[0,1]$ to $[\epsilon, 1-\epsilon]$ and that $g'(r) = 1-2\epsilon$. Define
\begin{equation*}
f(r)=v(g(1-r)), \qquad \nu(r)=\rho(g(1-r)).
\end{equation*}
We have that $f(0)=v(g(1))=v(1-\epsilon)=0$, hence $f$ so defined satisfies the hypotheses of Theorem \ref{steinerbergpoinc}. The weight $\rho$ is continuous and non-negative on $[0,1]$, and furthermore, from the computations in Remark \ref{underweightrmk}, for all $r\in[\epsilon, 1-\epsilon]$
\begin{equation*}
\rho(r) \geq \left(\frac{r}{2}\right)^{\frac{\Lambda}{\lambda}(n-1)} > 0,
\end{equation*}
i.e., $\rho$ is also non-vanishing.

By changing variables, we obtain
\begin{align*}
\int_{0}^{1} |f(r)|\nu(r) \,dx ={}& \int_{0}^{1} |v(g(1-r))| \rho(g(1-r)) \,dr\\
={}& \frac{1}{1-2\epsilon}\int_{\epsilon}^{1-\epsilon} |v(r)| \rho(r) \,dr,
\end{align*}
and similarly,
\begin{align*}
\int_{0}^{1} |f'(r)|\nu(r) \,dr ={}& \int_{0}^{1}|v'(g(r))||g'(r)|\rho(g(1-r)) \,dr\\
={}& \int_{\epsilon}^{1-\epsilon} |v'(r)| \rho(r) \,dr.
\end{align*}
Next, we estimate the constant in \eqref{steinerbergerineq}. It is easy to check that $\rho$ is strictly increasing from the definition, as is $g$. Hence, for $s\geq r$ we have $\rho(g(1-s)) \leq \rho(g(1-r))$. Therefore, for all $0\leq r\leq 1$,
\begin{align*}
\frac{1}{\nu(r)} \int_r^1\nu(s) \,ds ={}& \frac{1}{\rho(g(1-r))} \int_r^1 \rho(g(1-s))\,ds\\
\leq{}& \frac{1}{\rho(g(1-r))} \rho(g(1-r))(1-r)\\
\leq{}& 1-r \leq 1.
\end{align*}
%\max_{0\leq r\leq 1} \frac{1}{\nu(r)} \int_r^1\nu(s) \,ds ={}& \max_{0\leq r\leq 1} \frac{1}{\rho(g(1-r))} \int_r^1 \rho(g(1-s))\,ds\\
%\leq{}& \max_{0\leq r\leq 1} \frac{1}{\rho(g(1-r))} \rho(g(1-r))(1-r)\\
%\leq{}& \max_{0\leq r\leq 1} 1-r = 1.

By the above computations, we can apply \eqref{steinerbergerineq} to obtain
\begin{equation*}
\int_\epsilon^{1-\epsilon} |v(r)|\, \rho(r)\,dx \leq \int_\epsilon^{1-\epsilon} |v'(r)|\, \rho(r)\,dr.
\end{equation*}
\hfill\qed

\begin{remark}\label{discoweight}
	In the proof of Theorem \ref{mainonball} we will actually use Lemma \ref{weighedpoincare} with the weight $\tilde{\rho}$ instead of $\rho$. This is possible since
	\begin{equation*}
	\frac{1}{\Lambda}\rho \leq \tilde{\rho} \leq \frac{1}{\lambda}\rho.
	\end{equation*}
	The preceding argument, however, does not apply to $\tilde{\rho}$ directly since it is not continuous.
\end{remark}

\section{Nonexistence of global solutions and LOBC}\label{nonexistence}
We now prove our main result in the radial case, i.e., when the spatial domain is a ball and the initial data is radially symmetric. The proof of the result in a general domain follows more or less easily from the radial case.\\

\noindent\textit{Proof of Theorem \ref{mainonball}:} Our proof uses key ideas from that of Theorem 2.1 in \cite{souplet2002gradient}. Some care is required in choosing the constants appearing in our argument in the correct order. Specifically, we first choose $u_0$ large in an appropriate sense, then choose the regularization parameters sufficiently small. This amounts to making $\epsilon$ and $t_0$ approach $0$, although the actual limit is not taken. This difficulty is not present in \cite{souplet2002gradient}, since the solutions dealt with therein are classical and no regularization is needed.

Consider the differential inequality
\begin{align}
	&\dot{y}(t) \geq Cy(t)^p, \quad 0<t_0<t<t_1, \label{blowupODE} \\ 
	& y(t_0) = M_0,
\end{align}
where $C,\,M_0>0$. We can integrate \eqref{blowupODE} explicitly to obtain
\begin{equation*}
	0\leq y(t)^{1-p} \leq C(1-p)(t-t_0) + M_0^{1-p}.
\end{equation*}
Hence, $y(t)^{1-p} \rightarrow 0$ as $t \rightarrow t_0 + \frac{M_0^{1-p}}{C(p-1)}$. Since $1-p<0$, this implies $y(t) \rightarrow +\infty$. Alternatively, for a fixed $t_1>t_0$, blow-up occurs for $t<t_1$ provided we have 
\begin{equation}\label{choiceM0}
	M_0 > \left[C(p-1)(t_1-t_0)\right]^{-\frac{1}{p-1}}.
\end{equation}

So fix $T>0$ and assume that the viscosity solution $u$ of \eqref{modeleq} in $B_1(0)\times [0,T]$ with radial initial data $u_0\in C^1(\overline{B_1(0)})$ satisfies \eqref{boundarydata} in the classical sense. Writing $u=u(r,t)$, with $r\in[0,1]$ and $t\in[0,T]$, this means $u(1,t)=0$ for all $t\in[0,T]$. We will specify the largeness condition on $u_0$ in terms of $M_0$ later, but may consider it set from now on, since it depends only on constants already available. 

Recall now the regularized function $w$ defined in Remark \ref{notationremark}, which satisfies the inequality \eqref{divformeq} in $(\epsilon, 1-\epsilon)\times (t_0,t_1)$. We take $\epsilon$ so that $0<\epsilon<\delta$, where $\delta$ is the same constant given above. This is so that $[\delta, 1-\delta]\subset (\epsilon, 1-\epsilon)$. Note also that the regularization in time may be performed so that $t_0$ and $t_1$ are arbitrarily close to $0$ and $T$, respectively (see Remarks \ref{notationremark} and \ref{radialrmk}). It follows that $M_0$ depends only on $p,\ T,$ and the coefficient $C$ in \eqref{blowupODE}. 

Using the solution pair $(\lambda_1, \varphi_1)$ to the eigenvalue problem \eqref{puccieigenvalue}, we define
\begin{equation*}
	z(t)=\int_{\epsilon}^{1-\epsilon} w(r,t) \varphi_1(r) \tilde{\rho}(r) \,dr, \quad t\in(t_0,t_1).
\end{equation*}
We will to show that $z=z(t)$ satisfies \eqref{blowupODE} with $z(t_0)\geq M_0$, and consequently blows up for some $t<t_1$. This is a contradiction, since $z$ is uniformly bounded for all $t\geq t_0$ by the uniform convergence of $w \rightarrow u$ and the fact that $u\ ,\varphi_1,$ and $\tilde{\rho}$ are all uniformly bounded (see Remark \ref{uniformboundu}, \eqref{upboundphi} and \eqref{overweight}, respectively). Therefore, the solution $u$ cannot satisfy the boundary data in the classical sense for all time. In other words, LOBC occurs.

Using \eqref{divformeq}, we compute
\begin{align}\label{zdot}
	\dot{z}(t) &={} \int_{\epsilon}^{1-\epsilon} w_t (r,t) \varphi_1(r) \tilde{\rho} \,dr \geq \int_{\epsilon}^{1-\epsilon} \left((\rho w')' + \tilde{\rho}|w'|^p\right) \varphi_1(r) \,dr\nonumber\\
	&={} \int_{\epsilon}^{1-\epsilon} (\rho w')'\varphi_1(r) \,dr + \int_{\epsilon}^{1-\epsilon} |w'|^p \varphi_1(r) \tilde{\rho} \,dr\nonumber\\
	&=:{} I_1 + I_2.
\end{align}

(We omit some of the functions' arguments for simplicity.) Integrating by parts twice in $I_1$, we obtain
\begin{align*}
	I_1 = \int_{\epsilon}^{1-\epsilon} w(r,t)\,(\rho \varphi_1')' \,dr + \rho w'\varphi_1|_\epsilon^{1-\epsilon} - \rho w\varphi_1'|_\epsilon^{1-\epsilon}.
\end{align*}
Since $\varphi_1(\epsilon)=\varphi_1(1-\epsilon)=0$, we have that $\rho w'\varphi_1|_\epsilon^{1-\epsilon} = 0.$ On the other hand, $\varphi_1'(\epsilon)>0$, $\varphi_1'(1-\epsilon)<0$ and $\rho, w\geq 0$ imply that $- \rho w\varphi_1'|_\epsilon^{1-\epsilon} \geq 0.$ Hence, we continue estimating
\begin{align}\label{i1}
	I_1 \geq{}& \int_{\epsilon}^{1-\epsilon} w(r,t)\,(\rho \varphi_1')'  \,dr \geq \int_{\epsilon}^{1-\epsilon} w(r,t)\,(\tilde{\rho} \Mcalm(D^2\varphi_1))  \,dr\nonumber\\
	={}& \int_{\epsilon}^{1-\epsilon} w(r,t)\,(-\lambda_1\varphi_1) \tilde{\rho}  \,dr = -\lambda_1 z(t).
\end{align}
The second inequality above comes from the minimality of the Pucci operator. Indeed, for all radial $\varphi\in C^2$, by the definition of the weights $\rho$ and $\tilde{\rho},$ we have
\begin{equation*}
	\frac{1}{\tilde{\rho}}(\rho\,\varphi')' = \theta(w'')\varphi'' + \theta(w')\varphi'\frac{n-1}{r}
\end{equation*}
wherever $w''$ is defined. This defines an elliptic operator with ellipticity constants $\lambda, \Lambda$. Therefore, $\Mcalm(D^2\varphi) \leq \nicefrac{1}{\tilde{\rho}}(\rho\,\varphi')'$ a.e. in $(\epsilon, 1-\epsilon)$.

We turn to estimating $I_2$. From Hölder's inequality for the measure $\tilde{\rho}(r) \,dr$,
\begin{align}\label{pgt2}
	\int_{\epsilon}^{1-\epsilon} |w'|\, \tilde{\rho} \,dr ={}& \int_{\epsilon}^{1-\epsilon} |w'| (\varphi_1)^\frac{1}{p} (\varphi_1)^{-\frac{1}{p}}\tilde{\rho} \,dr\nonumber\\
	\leq{}& \left(\int_{\epsilon}^{1-\epsilon} (\varphi_1)^{-\frac{1}{p-1}} \tilde{\rho} \,dr\right)^\frac{p-1}{p} \left(\int_{\epsilon}^{1-\epsilon} |w'|^p \varphi_1\tilde{\rho} \,dr\right)^\frac{1}{p}.
\end{align}
The assumption $p>2$ implies $\nicefrac{1}{p-1} \in (0,1)$, so we can apply Lemma \ref{eigenlemma} to bound the first integral in the right-hand side of \eqref{pgt2} by a constant $C$ that does not depend on $\epsilon$. We then have
\begin{equation}\label{fromhoel}
	\int_{\epsilon}^{1-\epsilon} |w'| \tilde{\rho} \,dr \leq C \left(\int_{\epsilon}^{1-\epsilon} |w'|^p \varphi_1\tilde{\rho} \,dr\right)^\frac{1}{p} = CI_2^\frac{1}{p}.
\end{equation}

Define $\tilde{w}(r,t):= w(r,t) - w(1-\epsilon)$ for all $r\in[\epsilon, 1-\epsilon],\ t\in (t_0,t_1)$. Note that $\tilde{w}'=w'$, and $\tilde{w}(1-\epsilon, t)=0$, hence Poincaré's inequality (Lemma \ref{weighedpoincare}) applies. Moreover, since we are taking $u_0\geq 0,\, u_0\not\equiv 0 $,  the strong minimum principle for $u$ the viscosity solution of \eqref{modeleq} in $B_1(0)\times [0,T]$ implies that $u(0,t)>0$ for all $t>0$ (see e.g., \cite{da2004remarks}). Therefore, by the uniform convergence of $w \to u$, we have that
\begin{equation*}
	w(\epsilon, t) \to u(0,t) > 0, \quad\text{for all } t>t_0,\ \textrm{as } \epsilon\to 0.
\end{equation*}
On the other hand, the uniform convergence $w\to u$ and the uniform continuity of $u$ in $\overline{B_1(0)}\times[0,T]$ imply that  $w(1-\epsilon) \to 0$ as $\epsilon\to 0$. Indeed, let $\nu > 0$. By assumption $u(1,t)=0$ for all $t\geq t_0$. Hence, for small $\epsilon>0$,
 \begin{align}\label{w1-eps}
 	w(1-\epsilon,t) ={}& w(1-\epsilon,t) - u(1,t)  \leq |w(1-\epsilon,t) - u(1-\epsilon,t)| \nonumber\\
 	&{}\ + |u(1-\epsilon,t) - u(1,t)| < 2\nu.
 \end{align}
(We will later simply write $w(1-\epsilon)=o(1)$.) Again by the minimum principle ($w$ is an $L^\infty$-strong solution, as shown in Proposition \ref{appeqlemma}, hence also a viscosity solution; see e.g., \cite{crandall1996equivalence}),
\begin{equation*}
	\min_{[\epsilon, 1-\epsilon]} w(\cdot, t) = \min\{w(\epsilon,t), w(1-\epsilon, t)\}.
\end{equation*}
Together with the considerations above, this implies that $\min_{[\epsilon, 1-\epsilon]} w(\cdot, t) = {w(1-\epsilon, t)}$. Therefore, $\tilde{w} \geq 0$.

Thus, applying Lemma \ref{weighedpoincare},
\begin{equation*}
	\int_{\epsilon}^{1-\epsilon} (w(r,t) - w(1-\epsilon,t))\tilde{\rho}\,dr = \int_{\epsilon}^{1-\epsilon} \tilde{w}(r,t)\,\tilde{\rho}\,dr \leq \int_{\epsilon}^{1-\epsilon} |\tilde{w}'|\,\tilde{\rho}\,dr = \int_{\epsilon}^{1-\epsilon} |w'|\,\tilde{\rho}\,dr.
\end{equation*}
Since $\tilde{\rho}$ is uniformly bounded (see \eqref{overweight}), this gives
\begin{equation*}
	\int_{\epsilon}^{1-\epsilon} w(r,t)\,\tilde{\rho}\,dr \leq Cw(1-\epsilon,t) + \int_{\epsilon}^{1-\epsilon} |w'|\,\tilde{\rho}\,dr.
\end{equation*}
Recalling \eqref{upboundphi} and using the elementary inequality ${(a+b)^p\leq 2^{p-1}(a^p + b^p)}$, we have
\begin{align*}
	z(t)^p ={}& \left(\int_{\epsilon}^{1-\epsilon} w(r,t)\,\varphi_1(r)\,\tilde{\rho}\,dr\right)^p \leq C\left(\int_{\epsilon}^{1-\epsilon} w(r,t)\,\tilde{\rho}\,dr\right)^p\\
	\leq{}& C\left(Cw(1-\epsilon,t) + \int_{\epsilon}^{1-\epsilon} |w'|\,\tilde{\rho}\,dr\right)^p\\
	\leq{}& C\left[w(1-\epsilon,t)^p + \left(\int_{\epsilon}^{1-\epsilon} |w'|\,\tilde{\rho}\,dr\right)^p\right].
\end{align*}
Together with \eqref{fromhoel}, this implies
\begin{equation}\label{i2}
	I_2 \geq C\left(\int_{\epsilon}^{1-\epsilon} |w'|^p \varphi_1\tilde{\rho} \,dr\right)^p \geq Cz(t)^p - w(1-\epsilon,t)^p.
\end{equation}

Thus, combining \eqref{zdot}, \eqref{i1}, \eqref{w1-eps} and \eqref{i2}, we have obtained
\begin{equation}\label{zODE1}
	\dot{z}(t) \geq - \lambda_1 z(t) + Cz(t)^p + o(1), \quad t\in (t_0,\,t_1),
\end{equation}
where, crucially, the coefficient $C$ does not depend on either $\epsilon$ or $u_0$.

We can reduce \eqref{zODE1} to \eqref{blowupODE} as follows. Using that $\varphi_1 = \varphi_1^\epsilon \to \hat{\varphi}$ in $[\delta, 1-\delta]$ uniformly as $\epsilon\to 0$ (see Lemma \ref{unilowerboundforphi}), $w \to u$ in $[0,1]\times [0,T]$ uniformly as $\epsilon,\ t_0\to 0$, and the uniform continuity of $u$ (more precisely that $u(\cdot, t_0)\to u_0$ as $t_0\to 0$), the bound $\tilde{\rho}\geq \hat{\rho}$ given in \eqref{underweight}, and the fact that all these functions are nonnegative, we have
\begin{align}\label{estimzinit}
	z(t_0) ={}& \int_{\epsilon}^{1-\epsilon} w(r,t_0)\, \varphi_1(r)\, \tilde{\rho}(r) \,dr \geq
	\int_{\delta}^{1-\delta} w(r,t_0)\, \varphi_1(r) \hat{\rho}(r)\,dr \nonumber\\
	\geq{}&  \int_{\delta}^{1-\delta} w(r,t_0)\, \hat{\varphi}(r) \hat{\rho}(r)\,dr  + o(1)\nonumber\\
	\geq{}&  \int_{\delta}^{1-\delta} u_0(r) \hat{\varphi}(r) \hat{\rho}(r)\,dr + o(1),
\end{align}
where we have also used that $\hat{\varphi}$ and $\hat{\rho}$ are uniformly bounded. Since these functions are also bounded by below in $[\delta, 1-\delta]$ by a positive constant, choosing the value of the last integral in \eqref{estimzinit} is equivalent to the condition \eqref{intconditionu0} from the statement of the Theorem.

%By the uniform convergence of $w \to u$ in $[0,1]\times [0,T]$ and the uniform convergence $\varphi_1 = \varphi_1^\epsilon \to \hat{\varphi}$ in $[\delta, 1-\delta]$, we have
%\begin{align}\label{estimzinit}
%	z(t_0) ={}& \int_{\epsilon}^{1-\epsilon} w(r,t_0)\, \varphi_1(r)\, \tilde{\rho}(r) \,dr \geq
%	\int_{\delta}^{1-\delta} w(r,t_0)\, \varphi_1(r) \hat{\rho}(r)\,dr \nonumber\\
%	\geq{}& \int_{\delta}^{1-\delta} w(r,t_0)\, [\hat{\varphi}(r)-\mu] \hat{\rho}(r)\,dr = \int_{\delta}^{1-\delta} w(r,t_0)\, \hat{\varphi}(r) \hat{\rho}(r)\,dr  + o(1)\nonumber\\
%	\geq{}& \int_{\delta}^{1-\delta} [u_0(r)-\mu]\, \hat{\varphi}(r) \hat{\rho}(r)\,dr + o(1)\geq  \int_{\delta}^{1-\delta} u_0(r) \hat{\varphi}(r) \hat{\rho}(r)\,dr + o(1),
%\end{align}

We recall from the proof of Lemma \ref{unilowerboundforphi} that $\lambda_1=\lambda_1^\epsilon$ is also uniformly bounded. Since $p>2$, this implies that the term $Cz(t)^p$ dominates the linear term in \eqref{zODE1}. More precisely, if, say $\lambda_1^\epsilon \leq C'$, setting
\begin{equation*}
	\int_{\delta}^{1-\delta} u_0(r) \hat{\varphi}(r) \hat{\rho}(r)\,dr \geq \max\left\{M_0, \left(\frac{2C'}{C}\right)^{\frac{1}{p-1}}\right\}  + 1,
\end{equation*}
gives both $\dot{z}\geq (\nicefrac{C}{2})z(t)^p$ and $z(t_0)\geq M_0$, which is equivalent to $\eqref{blowupODE}$. This gives the desired contradiction.\hfill\qed

\begin{remark}
	The hypothesis which leads to contradiction, i.e, that the solution $u$ satisfies the boundary data in the classical sense, is used only to determine ${w(1-\epsilon,t)=o(1)}$ in $\eqref{w1-eps}$, and in the subsequent application of Lemma \eqref{weighedpoincare}. This is essential, however, to show that $z$ satisfies $\eqref{blowupODE}$. 
	
	Note also that, although $\eqref{blowupODE}$ blows-up for $p>1$, the use of $p>2$ is crucial in the application of Lemma \ref{eigenlemma} in the estimate \eqref{pgt2}.
\end{remark}

We now use Theorem \ref{mainonball} to provide an example of LOBC for solutions of \eqref{modeleq}-\eqref{initialdata} in a more general bounded domain. The computations closely follow \cite{yuxiang2010single}.
\begin{corollary}\label{gendomain}
	Let $\Omega$ be a bounded domain satisfying a uniform interior sphere condition. Then, there exist $u_0 \in C^1(\overline{\Omega}),$ with $u_0\geq 0$ and $u_0|_{\partial \Omega} = 0$, such that LOBC occurs for solutions of \eqref{modeleq}-\eqref{initialdata}  in a finite time $T=T(u_0, \Omega)$.
\end{corollary}

\textit{Proof:}
	From the interior sphere condition, there exists an $\eta > 0$ such that for all $x_0\in \partial \Omega$, there exists a ball of radius $\eta$ tangent to $\partial \Omega$ at $x_0$, say $B_\eta(x_1)$. Consider $\varphi\in C_0^\infty(B_1(0))$ a radial cut-off function such that 
	\begin{equation*}
		\varphi(r)=\left\{\begin{array}{rl}
			1, &\quad r \leq \frac{2}{3}\\[6pt]
			0, &\quad r \geq \frac{3}{4}
		\end{array}\right.
	\end{equation*}
	and consider the solution $v$ of 
	\begin{equation*}
		\left\{\begin{array}{ll}
			v_t - \Mcalm(D^2v) - |Dv|^p = 0 &\text{in } B_1(0)\times(0,\infty),\\
			v = 0 &\text{on } \partial B_1(0)\times[0,\infty),\\
			v(x,0) =  C\varphi(|x|) &\text{in } \overline{B_1(0)},
		\end{array}\right.
	\end{equation*}
	where the boundary condition is understood in the viscosity sense. It is easy to check that, for large enough $C>0$, $C\varphi$ satisfies \eqref{intconditionu0}. Hence, by Theorem \ref{mainonball}, LOBC occurs for $v$ at some time $T = T(C\varphi) > 0$. As before, $v$ is radial, thus $v(x, T(C\varphi)) > 0$ for all $x\in \partial B_1(0)$.
	
	We now rescale and translate $v$ to obtain a solution in $B_\eta(x_1)$: define
	\begin{equation}\label{rescaling}
		\tilde{v}(x,t) = \eta^k v(\eta^{-1}|x-x_1|, \eta^{-2}t),
	\end{equation}
	where $k=\frac{p-2}{p-1}$. Then $\tilde{v}$ is a solution of \eqref{modeleq} in $B_\eta(x_1)\times (0,\infty)$ satisfying homogeneous boundary data (again in the viscosity sense) and initial condition
	\begin{equation*}
		\tilde{v}(x,0) = C\eta^k \varphi(\eta^{-1}|x-x_1|).
	\end{equation*}
	We note that, since the rescaling \eqref{rescaling} produces a solution of \eqref{modeleq} on the corresponding rescaled domain, the boundary condition in the viscosity sense is preserved: if the equation holds ``up to a boundary point'' $(x,t)\in\partial B_1(0)\times(0,\infty)$, it will hold up to the point of $\partial B_\eta(x_1)\times (0,\infty)$ where it is mapped.
	
	The solution $\tilde{v}$ is radially symmetric with respect to $x_1$. Thus we have $\tilde{v}(x_0,T)>0$, where $T=\eta^2 T(C\varphi)$.
	
	Define now
	\begin{equation*}
		u_0(x) = \left\{\begin{array}{cl}
		\tilde{v}(x,0) &\quad\text{if } x\in B_\eta(x_1),\\
		0 &\quad\text{if } x\in \overline{\Omega}\backslash B_\eta(x_1),
		\end{array}\right.
	\end{equation*}
	and consider the solution $u$ of
	\begin{equation*}
		\left\{\begin{array}{ll}
			u_t - \Mcalm(D^2u) = |Du|^p &\text{in } \Omega\times(0,\infty),\\
			u = 0 &\text{on } \partial\Omega\times[0,\infty),\\
			u(x,0) = u_0 (x) &\text{in }\overline{\Omega},
		\end{array}\right.
	\end{equation*}
	with $u_0$ as previously defined. Of course, $u$ is also a solution of \eqref{modeleq} in $B_\eta(x_1) \times(0,\infty)$, and satisfies $u\geq 0$ on $\partial B_\eta(x_1) \times(0,\infty)$ in the viscosity sense. Thus, by comparison we have
	\begin{equation*}
		u \geq \tilde{v} \quad\text{in } \overline{B_\eta(x_1)} \times [0,\infty).
	\end{equation*}
	Hence,
	\begin{equation*}
		u(x_0, T) \geq \tilde{v}(x_0, T) > 0,
	\end{equation*}
	i.e., LOBC occurs for $u$.\hfill\qed

The previous result might be rephrased to include a condition applicable to more general $u_0$ than the example provided. We avoided this for simplicity, since the condition is rather convoluted, but do so now for completeness.

For any ball $B_\eta(x_1)\subset \Omega$, where $x_1\in\Omega$, $\eta>0$ and $\partial B_\eta(x_1)$ is tangent to $\partial\Omega$ at $x_0$, denote the radial variable by $r=| x - x_1 |$. Note $0<r<\eta$. For any $v$ defined in $B_\eta(x_1)$, we may define the radial symmetrization
\begin{equation*}
	s(v)(r) = \inf_{\partial B_r(x_1)} v.
\end{equation*}
Note that, for any $u_0\in C(\overline{\Omega})$ such that $u_0 |_{\partial \Omega} = 0$, we have $s(u_0)\leq u_0$ in $B_\eta(x_1)$ and $s(u_0)(\eta) = s(u_0)(x_0) = 0$. Note also that $\|s(u_0)\|_\infty = \|u_0\|_\infty.$

\begin{corollary}\label{gendomainbis}
	Using the notation above, as well as that of Theorem \ref{mainonball} and Corollary \ref{gendomain}, there exists positive constants $\delta=\delta(\lambda, \Lambda, n)$ and $M=M(\lambda, \Lambda, n, p)$ such that LOBC occurs for all solutions of \eqref{modeleq}-\eqref{initialdata} with initial data $u_0$ such that
	\begin{equation}\label{largenessgendomain}
		\sup \left\{\eta^{-k} \int_\delta^{1-\delta} s(u_0)(\eta r) \,dr \right\} > M.
	\end{equation}
	Here, the supremum runs over all $B_{\eta}(x_1) \subset \Omega$ tangent to $\partial \Omega$ for fixed $\eta>0$.
\end{corollary}

\textit{Proof:} The corresponding proof is analogous to that of Corollary \ref{gendomain}, in that it follows by a comparison argument and scaling between $B_\eta(x_1)$ and $B_1(0)$. Hence, it will be omitted. \hfill\qed

\begin{remark}
	After a change of variable, condition \eqref{largenessgendomain} can be written as
	\begin{equation*}
		\sup \left\{ \int_{\eta\delta}^{\eta(1-\delta)} s(u_0)(r) \,dr \right\} > \eta^{k+1}M,
	\end{equation*}
	which more closely resembles the condition given for Theorem \ref{mainonball}, in that the limits of integration and the constant on the right-hand side reflect the dependence on $\Omega$ (through $\eta$), in addition to $\lambda, \Lambda, n$ and $p$. Note that the supremum still runs over all interior tangent spheres $B_{\eta}(x_1)$ for fixed $\eta>0$.
\end{remark}

\section{Extensions}\label{extensions}

To extend the results regarding LOBC to more general equations, we must first guarantee that the global existence result of \cite{barles2004generalized} applies to the equations considered. In fact, that the result applies to our model equation, \eqref{modeleq}, will follow as a particular case. For convenience, we restate part of what was mentioned in the introduction. Namely, consider
\begin{equation*}
u_t - F(D^2 u) = f(Du) \textrm{ in } \Omega \times (0,T),
\end{equation*}
where the nonlinearities are as follows: $F:S(n) \rightarrow \mathbb{R}$ is uniformly elliptic, i.e., 
\begin{equation*}
\Mcalm(X - Y) \leq F(X) - F(Y) \leq \Mcalp (X - Y) \quad \textrm{for all } X,Y \in S(n),
\end{equation*}
and vanishes at zero. In particular, this implies that
\begin{equation}\label{extremalbounds}
\Mcalm(X) \leq F(X) \leq \Mcalp(X)\quad \textrm{for all} \ X\in S(n).
\end{equation}
The gradient nonlinearity $f:\mathbb{R}^n \rightarrow \mathbb{R}$ satisfies $f(\xi) \geq |\xi|^2 h(|\xi|)$ for all $\xi\in \mathbb{R}^n$, where $h:\mathbb{R}\rightarrow\mathbb{R}$ satisfies the growth condition \eqref{growthh} and
\begin{align}
	& h=h(s) \textrm{ is positive nondecreasing for } s>0,\label{hhyp1}\\
	& s\mapsto s^2h(s) \textrm{ is convex},\\
	& h(yz) \leq C(h(y) + h(z))  \textrm{ for large } y,z>0 \textrm{ and some } C>0.\label{hhyp3}
\end{align}
The last condition implies that $h$ grows more slowly than any positive power. Examples which satisfy the conditions above are $h(s) = (\log s)^p$ and $h(s) = (\log s)^p (\log \log s)^q$, for large $s$ and $p,q>0$.

\subsection{Comparison, existence and uniqueness}\label{comparisonsubsec}

We look to verify hypotheses (H1) and (H2) needed for the Strong Comparison Principle, Theorem \ref{SCR}. We begin by setting
\begin{equation}
	G(x,r,\xi,X) = G(\xi,X) = -F(X) + f(\xi),
\end{equation}
where $\xi\in\mathbb{R}^n$, $X\in S(n)$, and the nonlinearities $F, f$ (and consequently, $h$) are as above. Note that this is compatible with the exchange of sub- and supersolutions mentioned in Remark \ref{signchange}, since $\tilde{F}:S(n) \rightarrow \mathbb{R}$ given by
\begin{equation*}
	\tilde{F}(X) = - F(-X), \quad \textrm{ for all } X\in S(n)
\end{equation*}
is uniformly elliptic and vanishes at zero if $F$ does.

We proceed to check (i) through (iii) of property (P) for $h_1(s)=s^2 h(s)$, with $h$ as above:
\begin{enumerate}[(i)]
	\item By the growth condition \eqref{growthh}, we have
	\begin{equation*}
		\int_1^\infty \frac{s}{h_1(s)}\,ds = \int_1^\infty \frac{s}{s^2h(s)} = \int_1^\infty \frac{1}{sh(s)}\,ds < \infty.
	\end{equation*}
	\item Since $h$ is nondecreasing,
	\begin{equation*}
		L\mapsto L^2s^2 h(Ls) - CL^2s^2h(s) = L^2s^2 (h(Ls) - Ch(s))
	\end{equation*}
	is increasing for all $s>0$ and $L\geq 1$.
	\item Since $L,s>0$ will be taken large, it is equivalent to show that, for fixed $C,\tilde{C}>0$ and $\epsilon>0$,
	\begin{align*}
		& h(Ls) - Ch(s) \geq \epsilon,\\
		& \epsilon > \nicefrac{\tilde{C}}{Ls}.
	\end{align*}
	It follows from the growth condition \eqref{growthh} on $h$ that $h(s)\rightarrow\infty$ as $s\rightarrow\infty$. Hence, fixing $s>0$ and taking large enough $L=L(s)>1$, we get $h(Ls) \geq \epsilon + Ch(s).$ The second inequality above comes from choosing $L$ large as well.
\end{enumerate}

This shows \eqref{uniparaboleq} satisfies (H1). Now on to (H2). The second matrix inequality in (H2) implies $X\leq Y + o(1)$. Hence, $\mu X\leq Y + o(1)$ for any $0<\mu<1$, and from the uniform ellipticity and the definition of the Pucci operator, this gives
\begin{equation*}
	F(\mu X) - F(Y) \leq \Mcalp(\mu X - Y) \leq o(1).
\end{equation*}
For the contribution of the gradient term to the estimate of (H2), if suffices to have $h_1$ above (i.e., $s\mapsto s^2h(s)$) be locally Lipschitz and satisfy the following, as noted in Example 1 of \cite{barles2004generalized}:
\begin{enumerate}[(H1)]
	\setcounter{enumi}{2}
	\item For all $C>0$, there exists a sequence $0<\mu_\epsilon<1$ defined for $0<\epsilon\leq 1$ such that $\mu_\epsilon\rightarrow 1$ as $\epsilon\rightarrow 0$ and such that for all large $r>0$ large enough and $0<\epsilon$ small enough, we have:
	\begin{equation*}
		C\epsilon r \sup_{0\leq \tau \leq r(1+C\epsilon)} |h_1'(\tau)| \leq (1-\mu_\epsilon) \inf_{\tau\geq r(1-C(1-\mu_\epsilon))} (h_1'(\tau)\tau - h_1(\tau)).
	\end{equation*}
\end{enumerate}

Given the properties of $h$ above, a lengthy but straightforward computation shows that to verify (H3) it suffices to choose $\mu_\epsilon$ such that $\epsilon^{-1}(1-\mu_\epsilon)\rightarrow +\infty$ as $\epsilon\rightarrow 0$.

\subsection{Loss of boundary conditions}\label{subsLOBCextens}
What follows is an extension of Theorem \ref{mainonball} that includes a more general gradient term, with a suitable growth condition. Consider
\begin{equation}\label{gengradtermeq}
	u_t - \Mcalm(D^2 u) = g(|Du|) \ \textrm{ in } \ B_1(0) \times (0,T),
\end{equation}
where $g:\mathbb{R}\to\mathbb{R}$, $g$ is convex increasing for $s\geq 0$, and $g(0) = 0$. Note that \eqref{gengradtermeq} has no ``$(x,t)$-dependence'', so that Lemma \ref{infconvpreserves} applies directly. Lemma \ref{radiallemma} applies as well if $u_0$ is radially symmetric, given that $g=g(|Du|)$.

On the other hand, Proposition \ref{appeqlemma} requires a slight adaptation. In the case that $w(\hat{x},\hat{t}) < u_{\epsilon, \kappa}(\hat{x},\hat{t})$, where $(\hat{x},\hat{t})$ is a point of second-order differentiability of the regularized function $w$, we must take the regularization parameter $\delta = \delta(g)$ small enough so that
\begin{align*}
	w_t(\hat{x},\hat{t}) - \Mcalm(D^2w(\hat{x},\hat{t})) - g(|Dw(\hat{x},\hat{t})|) \geq{}& - \frac{K}{\kappa^\frac{1}{2}} + \frac{\lambda}{\delta} - \frac{\Lambda(n-1)}{\epsilon} - g\left(\frac{K}{\epsilon^\frac{1}{2}}\right)\\
	\geq{}& 0
\end{align*}
to get that $w$ is a supersolution.

For the statement of the following Lemma, we define
\begin{equation*}
	a(s) = \sup_{y>0} \frac{g^{-1}(ys)}{g^{-1}(y)}, \quad s\geq 0,
\end{equation*}
and recall the definition of the convex conjugate,
\begin{equation*}
	g^*(s) = \sup\,\{\,ys - g(y)\,|\, y \in \mathbb{R}\,\}.
\end{equation*}

\begin{lemma}\label{gengradtermlemma}
	Let $u_0\in C(\overline{B_1(0)})$ be a radial function, and $g$ as described above. Assume also that $g$ is such that \eqref{gengradtermeq} satisfies \emph{(H1)} and \emph{(H2)} of Section \ref{comparison}. If
	\begin{equation}\label{growthg*}
		\int_1^\infty \frac{g^*(La(s))}{s^2} \,ds < \infty
	\end{equation}
	for all $L>0$, then there exist positive constants $\delta$ and $M$, depending only on $\lambda, \Lambda, n$ and $g$, such that, if
	\begin{equation}\label{intconditiongengrad}
		\int_\delta^{1-\delta} u_0(r) - \frac{1}{2}\|u_0\|_\infty \,dr > M,
	\end{equation}	
	then the solution $u$ of \eqref{gengradtermeq}, \eqref{boundarydata}, \eqref{initialdata} with $\Omega = B_1(0)$ and initial data $u_0$ has LOBC at some finite time $T=T(u_0)$.
\end{lemma}

\textit{Proof:} Aside from using all the auxiliary results leading to Theorem \ref{mainonball}, the proof follows that of Theorem 5.2 in \cite{souplet2002gradient}. We repeat most of the argument for convenience. Once more, we proceed by contradiction, assuming $u$ is a solution which satisfies \eqref{boundarydata} in the classical sense. We consider $\varphi_1$ as previously defined, and again denote by $w$ the function obtained by regularizing the radial part of the solution $u$ of \eqref{gengradtermeq} for $\Omega=B_1(0)$. This function now satisfies, for arbitrary $\epsilon>0$ and $0< t_0 < t_1<T$,
\begin{equation*}
	\tilde{\rho} w_t \geq (\rho w')' + 	\tilde{\rho}g(|w'|) \quad \textrm{for a.e. } r\in(\epsilon,1-\epsilon), t\in (t_0,t_1).
\end{equation*}

From the definition of $a$, setting $y = g(|w'(r,t)|)\varphi_1(r,t)$ for ${r\in(\epsilon, 1-\epsilon)}$, ${ t_0<t<t_1}$, we have
\begin{equation*}
	a(\nicefrac{1}{\varphi_1}) g^{-1}(g(|w'|)\varphi_1) \geq g^{-1}(\nicefrac{y}{\varphi_1}) = |w'|.
\end{equation*}
Hence,
\begin{equation}\label{g-a-ineq}
	g(|w'|)\varphi_1 \geq g\left(\frac{|w'|}{a(\nicefrac{1}{\varphi_1})}\right).
\end{equation}
Let $L>0$ to be chosen later. By the definition of the convex conjugate, we have 
\begin{equation*}
	L|w'| = \frac{|w'|}{a(\nicefrac{1}{\varphi_1})}La(\nicefrac{1}{\varphi_1}) \leq g\left(\frac{|w'|}{a(\nicefrac{1}{\varphi_1})}\right) + g^*(La(\nicefrac{1}{\varphi_1}))
\end{equation*}
for a.e.~$r\in(\epsilon, 1-\epsilon), t_0<t<t_1$ (we have omitted the arguments for simplicity). This is an instance of Fenchel's inequality, analogous to that of Hölder's inequality in the proof of Theorem \ref{mainonball}.

Using \eqref{g-a-ineq}, we have
\begin{equation*}
	L\int_{\epsilon}^{1-\epsilon} |w'| \tilde{\rho} \,dr \leq \int_{\epsilon}^{1-\epsilon} g(|w'|)\varphi_1 \tilde{\rho} \,dr + \int_{\epsilon}^{1-\epsilon} g^*(La(\nicefrac{1}{\varphi_1})) \tilde{\rho} \,dr.
\end{equation*}
That the second integral is finite follows from \eqref{growthg*}. It can also be proven that it is bounded by a constant C that does not depend on $\epsilon$, as in Lemma \ref{eigenlemma}.

Arguing as in the beginning of the proof of Theorem \ref{mainonball}, we obtain
\begin{equation*}
	\dot{z}(t) + \lambda_1 z(t) \geq I
\end{equation*}
where $z$ is defined exactly as before, but now
\begin{align*}
	I :={}& \int_{\epsilon}^{1 - \epsilon} g(|w'|) \varphi_1 \tilde{\rho} \,dr\\
	\geq{}& L \int_{\epsilon}^{1 - \epsilon} |w'|\, \tilde{\rho} \,dr - \int_{\epsilon}^{1-\epsilon} g^*(La(\nicefrac{1}{\varphi_1})) \tilde{\rho} \,dr\\
	\geq{}& L \int_{\epsilon}^{1 - \epsilon} |w'|\, \tilde{\rho} \,dr - C\\
	\geq{}& L \left(\int_{\epsilon}^{1 - \epsilon} |w|\, \tilde{\rho} \,dr - w(1-\epsilon, t) \right) - C,
\end{align*}
where the last inequality is follows by applying Lemma \ref{weighedpoincare}. Setting $L$ sufficiently large, in terms of $\|\varphi_1\|_\infty$ and $\lambda_1$ (both of which are independent of $\epsilon$), and noting that $w(1-\epsilon, t)=o(1)$ as $\epsilon\to 0$, as before, we obtain
\begin{equation}\label{EDOgengrad} 
	\dot{z}(t) \geq z(t) - C \quad \textrm{for a.e. } t\in(t_0,t_1),
\end{equation}
which we can integrate to get
\begin{equation}\label{EDOgengradint} 
	z(t) \geq (z(t_0) - C)e^{t - t_0} \quad \textrm{for all } t\in(t_0,t_1).
\end{equation}

To conclude, note that \eqref{EDOgengrad} does not blowup in finite time, as does \eqref{blowupODE}. We will find a contradiction in the form of a bound, derived from \eqref{EDOgengradint}, which we can easily violate by choosing the appropriate $u_0$. Also, since our aim is to prove nonexistence beyond some finite time,  we may assume $T>0$ is large to achieve this contradiction. By choosing the time-regularization parameter small as well, the difference $t_1-t_0$ can be made large as well, say $t_1-t_0\geq\nicefrac{T}{2}$.

As in \eqref{estimzinit}, we have
\begin{equation}\label{estimminitzext}
	z(t_0) \geq \int_{\delta}^{1-\delta} u_0(r) \hat{\varphi}(r) \hat{\rho}(r)\,dr + o(1), \quad \textrm{as }\epsilon,\, t_0\to 0.
\end{equation}
On the other hand, by evaluating \eqref{EDOgengradint} at $t=t_1$,
\begin{align*}
	z(t_0) \leq{}& e^{-(t_1-t_0)} z(t) + C\\
	\leq{}& e^{-\frac{T}{2}}\int_{\epsilon}^{1-\epsilon} w(r,t)\varphi_1(r)\tilde{\rho} \,dr + C.
\end{align*}
Recalling the bounds for $\varphi_1,\,\tilde{\rho}$, that $\hat{\varphi},\,\hat{\rho}$ are bounded by below by a positive constant in $[\delta, 1-\delta]$, and that $\|w\|_\infty\leq\|u_0\|_\infty$, we take $T$ sufficiently large so that
\begin{equation*}
	z(t_0) \leq \frac{1}{2} \int_{\delta}^{1-\delta} \|u_0\|_\infty\hat{\varphi}\hat{\rho}\,dr + C.
\end{equation*}
Combining both estimates for $z(t_0)$, we have
\begin{equation*}
	\int_\delta^{1-\delta} (u_0 - \frac{1}{2}\|u_0\|_\infty) \hat{\varphi}(r)\hat{\rho}(r) \,dr < C,
\end{equation*}
where $C=C(\lambda, \Lambda, n, g)$. This bound is readily violated by choosing a suitably large $u_0$, and as before, it is equivalent to the one in the statement of the lemma since $\hat{\varphi}$ and $\hat{\rho}$ are uniformly bounded from below in $[\delta, 1-\delta]$. \hfill\qed

Finally, we proceed with the proof of the extension mentioned in the introduction.
\noindent\textit{Proof of Theorem \ref{uniparabolthm}:} First, define $g(s) = s^2 h(s)$, where $h$ satisfies {\eqref{hhyp1}-\eqref{hhyp3}} and the growth condition \eqref{growthh}. It is proven in Lemma 5.3 and the Completion of Theorem 2.2 in \cite{souplet2002gradient} that $g$ so defined satisfies the hypothesis of Lemma \ref{gengradtermlemma}, including \eqref{growthg*}. Furthermore, using \eqref{extremalbounds}, if $u$ is a solution of \eqref{uniparaboleq}, formally we have that
\begin{align*}
	u_t - \Mcalm(D^2 u) \geq u_t - F(D^2 u) = f(Du) \geq |Du|^2 h(|Du|) ={}& g(|Du|)\\
	&{}\textrm{ in } \Omega \times (0,T),
\end{align*}
and this is readily checked using test functions. Hence, $u$ is a supersolution of \eqref{gengradtermeq} in $\Omega\times (0,T)$. Using the interior sphere condition, without loss of generality we may assume that $B_1(0)\subset \Omega$ with $B_1(0)$ tangent to $\partial \Omega$ at some point $x_0\in\partial \Omega \cap \partial B_1(0)$ (This is equivalent to repeating the constructions of Theorems \ref{mainonball} and \ref{gengradtermlemma} on a ball of arbitrary radius and performing translation.) Arguing as in the proof of Corollary \ref{gendomain}, we conclude that $u$ is a supersolution of \eqref{gengradtermeq} in $B_1(0)\times (0,T)$. Let $\tilde{u}_0 \in C(\overline{B_1(0)})$ nonnegative, radially symmetric and satisfies \eqref{intconditiongengrad}, and consider the solution $\tilde{u}$ of \eqref{gengradtermeq} with initial data $\tilde{u}_0$ and homogeneous boundary data. By Lemma \ref{gengradtermlemma}, $\tilde{u}$ has LOBC in finite time $T'>0$, and since it is radially symmetric we may conclude that LOBC occurs at $x_0\in B_1(0)$. Now define
\begin{equation*}
	u_0(x) =\left\{\begin{array}{rl}
	\tilde{u}_0(x), &\quad x\in B_1(0)\\[6pt]
	0, &\quad x\in \overline{\Omega}\backslash B_1(0).
	\end{array}\right.
\end{equation*}
By comparison, we have that the solution $u$ of \eqref{uniparaboleq}-\eqref{boundarydata}-\eqref{initialdata} with initial data $u_0$ satisfies $u(x_0,T')\geq\tilde{u}(x_0,T')>0$, hence $u$ has LOBC.\hfill\qed

\begin{remark}
	The more general nonlinearities do not admit a rescaling argument like the one given in the proof of Corollary \ref{gendomain}. Furthermore, \eqref{uniparaboleq} may no longer have radial symmetry. The preceding argument is in some sense simpler, but it does not provide a condition one can check for any given initial data like the one in Corollary \ref{gendomainbis}.
\end{remark}

\begin{remark}
	The typical example for the nonlinearity $h$ in Theorem \ref{uniparabolthm} is $h(s) = (\log s)^q$ for $q>0$ and large $s$. In this case, the growth condition \eqref{growthh} forces that $q>1$. This is consistent with what is known to be a more precise condition for preventing GBU in the case of the viscous Hamilton-Jacobi equation: for
	\begin{equation*}
		u_t - \Delta u = f(u,\nabla u) \quad \textrm{ in } \Omega\times (0,T),
	\end{equation*}
	GBU does not occur if
	\begin{equation*}
		|f(u,\nabla u)| \leq C(u)(1 + |\nabla u|^2)h(|\nabla u|)
	\end{equation*}
	where $C(u)$ is locally bounded, and $h$ is positive nondecreasing and satisfies
	\begin{equation*}
		\int_1^\infty \frac{1}{sh(s)} = \infty.
	\end{equation*}
	See \cite{quittner2007superlinear}, Ch. IV, and the references therein.
\end{remark}

\begin{remark}\label{rmkextSphereUni}
	The proof of Theorem \ref{uniparabolthm} uses only the uniform interior sphere condition. Both interior and exterior sphere conditions are assumed to establish a connection between Theorems \ref{localexthm} and \ref{uniparabolthm}. Specifically, to have both results applicable in the same situation. See also Remarks \ref{intSphereLocal}.
\end{remark}

\noindent\textbf{Acknowledgments:} A.Q.~was partially supported by Fondecyt Grant No. 1151180, Programa Basal, CMM, U. de Chile and Millennium Nucleus Center for Analysis of PDE NC130017. A.R.~was partially supported by Conicyt, Beca Doctorado Nacional 2016, and Programa de Iniciación a la Investigación Científica (PIIC) 2015, Universidad Técnica Federico Santa María.

\bibliographystyle{plain}
\bibliography{aerp.bib}

\vspace{5mm}
\noindent \textsc{Alexander Quaas}\\
\noindent \textit{Email:} alexander.quaas@usm.cl\\
\noindent \textsc{Andrei Rodríguez}\\
\noindent \textit{Email:} andrei.rodriguez.14@sansano.usm.cl\\[4pt]
\noindent \textsc{Departamento de Matemática, Universidad Santa María, Casilla: V-110, Avda. España 1680, Valparaíso, Chile.}

\end{document}